\documentclass[12pt,a4paper,reqno]{amsart}
\usepackage[a4paper,left=16mm,right=16mm,top=36mm,bottom=36mm]{geometry}

\newtheorem{dfntn}{Definition}[section] 
\newtheorem{rmrk}{Remark}[section]    
\newtheorem{thrm}{Theorem}[section]     
\newtheorem{lmm}{Lemma}[section]	

\usepackage{amssymb}
\usepackage{amsmath}
\usepackage{graphicx}
\usepackage{float}
\usepackage[utf8]{inputenc}
\usepackage[english]{babel}
\usepackage{stmaryrd}
\newcommand{\R}{\mathbb{R}}

\newcommand{\eps}{\varepsilon}
\newcommand{\fhi}{\varphi}

\newcommand{\vertiii}[1]{{\left\vert\kern-0.25ex\left\vert\kern-0.25ex\left\vert #1
    \right\vert\kern-0.25ex\right\vert\kern-0.25ex\right\vert}}

\def\calA{\mathcal{A}}
\def\calB{\mathcal{B}}
\def\calE{\mathcal{E}}
\def\calG{\mathcal{G}}
\def\calL{\mathcal{L}}
\def\calO{\mathcal{O}}
\def\calH{\mathcal{H}}
\def\calT{\mathcal{T}}

\def\calF{\mathcal{F}}

\def\calR{\mathcal{R}}
\def\calZ{\mathcal{Z}}

\usepackage{caption}
\usepackage{subcaption}

\numberwithin{equation}{section}

\usepackage{tikz}
\usepackage{pgfplots}
\pgfplotsset{compat=1.15}

\usepackage{hyperref}

\begin{document}

\title[DG and $C^0$-IP FEM for periodic HJBI problems and numerical effective Hamiltonians]{Discontinuous Galerkin and $C^0$-IP finite element approximation of periodic Hamilton--Jacobi--Bellman--Isaacs problems with application to numerical homogenization}

\author[E. L. Kawecki]{Ellya L. Kawecki}
\address[Ellya L. Kawecki]{University College London, Department of Mathematics, Gower Street, London WC1E 6BT, UK}
\email{e.kawecki@ucl.ac.uk}

\author[T. Sprekeler]{Timo Sprekeler}
\address[Timo Sprekeler]{University of Oxford, Mathematical Institute, Woodstock Road, Oxford OX2 6GG, UK.}
\email{sprekeler@maths.ox.ac.uk}

\date{\today}

\begin{abstract}
In the first part of the paper, we study the discontinuous Galerkin (DG) and $C^0$ interior penalty ($C^0$-IP) finite element approximation of the periodic strong solution to the fully nonlinear second-order Hamilton--Jacobi--Bellman--Isaacs (HJBI) equation with coefficients satisfying the Cordes condition. We prove well-posedness and perform abstract \textit{a posteriori} and \textit{a priori} analyses which apply to a wide family of numerical schemes. These periodic problems arise as the corrector problems in the homogenization of HJBI equations. The second part of the paper focuses on the numerical approximation to the effective Hamiltonian of ergodic HJBI operators via DG/$C^0$-IP finite element approximations to approximate corrector problems. Finally, we provide numerical experiments demonstrating the performance of the numerical schemes.
\end{abstract}

\maketitle

\section{Introduction}

In the first part of this paper we study the periodic boundary value problem for the fully nonlinear second-order Hamilton--Jacobi--Bellman--Isaacs (HJBI) equation  
\begin{align}\label{Intro1}
\left\{\begin{aligned}
\inf_{\alpha\in\calA}\sup_{\beta\in\calB}\left\{ -A^{\alpha\beta}:\nabla^2 u - b^{\alpha\beta}\cdot \nabla u + c^{\alpha\beta} u -f^{\alpha\beta}\right\} = 0 \quad \text{in }Y,\\ u \text{ is $Y$-periodic},
\end{aligned}\right.
\end{align}
where $\calA$ and $\calB$ are compact metric spaces, and $Y:=(0,1)^n\subset\R^n$ denotes the unit cell in dimension $n\geq 2$. Here, we use the notation
\begin{align*}
\fhi^{\alpha\beta}:=\fhi(\,\cdot\,,\alpha,\beta),\qquad \fhi\in\{A,b,c,f\},
\end{align*}
and assume that the functions 
\begin{align*}
A=(a_{ij}):\R^n\times\calA\times\calB\rightarrow \R^{n\times n}_{\mathrm{sym}},\quad b=(b_i):\R^n\times\calA\times\calB\rightarrow \R^{n},\quad c,f:\R^n\times\calA\times\calB\rightarrow \R
\end{align*}
are uniformly continuous and $Y$-periodic in their first argument $y\in\R^n$. Further, we assume that $A$ is uniformly elliptic (see \eqref{unifm ell}), that $\inf_{\R^n\times\calA\times\calB} c>0$, and that the coefficients $A,b,c$ satisfy the Cordes condition 
\begin{align*}
\lvert A\rvert^2+\frac{\lvert b\rvert^2}{2\lambda} + \frac{c^2}{\lambda^2}\leq \frac{1}{n+\delta}\left(\mathrm{tr}( A) + \frac{c}{\lambda}  \right)^{2}   
\end{align*}
in $\R^n\times \calA\times\calB$ for some constants $\lambda>0$ and $\delta\in (0,1)$. These assumptions guarantee the existence and uniqueness of a periodic strong solution $u\in H^2_{\mathrm{per}}(Y)$ to the HJBI problem \eqref{Intro1}; see Section \ref{Sec wellp}.

The goal of the first part of the paper is the construction of discontinuous Galerkin (DG) and $C^0$ interior penalty ($C^0$-IP) finite element schemes for the periodic HJBI problem \eqref{Intro1} and their rigorous \textit{a posteriori} and \textit{a priori} error analysis; see Section \ref{Section FE schemes for periodic HJBI}.

The fully nonlinear HJBI equation is a very general elliptic PDE arising in many contexts, such as stochastic differential games and optimal control problems. In the case that one of the metric spaces $\mathcal{A},\mathcal{B}$ is a singleton set, the HJBI equation becomes the HJB equation arising in stochastic optimal control theory, with applications in finance, engineering, and renewable energies. Interestingly, the HJBI equation is capable of capturing other famous nonlinear PDEs, such as the fully nonlinear Monge--Amp\`{e}re (MA) equation arising in illumination optics, optimal transport (see Kawecki, Lakkis, Pryer \cite{KLP18}), and differential geometry. The MA equation is conditionally elliptic, with classical examples exhibiting a lack of uniqueness. The HJBI formulation of the MA equation is uniquely solvable and has been used in Feng, Jensen \cite{FJ17}, Brenner, Kawecki \cite{BK21} to overcome this lack of uniqueness.

The HJBI problem is well understood in the framework of viscosity solutions (see e.g., Fleming, Soner \cite{FS06}, Crandall, Ishii, Lions \cite{CIL92} and Ishii \cite{Ish89}), and there have been several numerical advances based on methods that enjoy a numerical analogue of the comparison principle used in the theory of viscosity solutions. Such methods include finite difference and semi-Lagrangian schemes such as Feng, Jensen \cite{FJ17}, and also integro-differential finite element methods; see Camilli, Jakobsen \cite{CJ09}, Salgado, Zhang \cite{SZ19}. However, enforcing a discrete maximum principle can be restrictive in practice and can lead to the requirement for large, or even unbounded stencils.

There is not a lot of work on finite element methods for periodic HJB/HJBI problems in the numerical analysis literature, and we refer to Gallistl, Sprekeler, S\"{u}li \cite{GSS21} for a mixed finite element scheme for periodic HJB problems. In recent years, there have been several advances in finite element methods for the Dirichlet problem based on the theory of the concept of strong solutions to HJBI equations. Such methods are typically more flexible than the finite difference method and allow one to capture complex geometries and to obtain higher order convergence rates. The existence and uniqueness of strong solutions to linear nondivergence-form PDEs (arising in the linearization of HJBI problems) and to the HJB equation was established in Smears, S\"{u}li \cite{SS13,SS14,SS16}, along with the well-posedness of optimal $hp$-finite element methods. These methods involved additional stabilizing forms that enforced a numerical analogue of the Miranda--Talenti estimate which is key to the well-posedness of the strong PDE. Other primal finite element methods that tackle the HJB problem are Neilan, Wu \cite{NW19} and Brenner, Kawecki \cite{BK21}. Here the authors use a discrete analogue of the Miranda--Talenti estimate, based on the theory of $H^2(\Omega)\cap H^1_0(\Omega)$ enrichment operators (see Neilan, Wu \cite{NW19}, Brenner, Kawecki \cite{BK21}, Kawecki, Smears \cite{KS21,KSm21}), to prove strong monotonicity of the scheme without the need for an additional stabilizing bilinear form. 

Following on from these approaches, these ideas have been extended from the HJB problem to the HJBI problem in Kawecki, Smears \cite{KS21}, and have been analyzed under a general framework that incorporates \textit{a priori} and \textit{a posteriori} error analysis for a wide family of finite element methods that encompasses the aforementioned schemes \cite{SS13,SS14,NW19,BK21}. In Kawecki, Smears \cite{KSm21}, the convergence of a family of adaptive finite element schemes for HJBI problems was proven. More recently, a virtual element method for the approximation of linear nondivergence-form PDEs and HJBI problems has been proposed and analyzed in Kawecki, Pryer \cite{KP21}.

Alongside this, we refer the reader to the papers \cite{CF95,CJ09,GSS21,GS19,Jen17,JS13} by various authors for finite element approaches allowing the use of $H^1$-conforming finite elements for HJB problems.
 
We refer to Kawecki \cite{Kaw19} for finite element methods for linear nondivergence-form elliptic PDEs on curved domains, and to Gallistl \cite{Gal19}, Kawecki \cite{Kaa19} for those with oblique boundary conditions. For a survey on recent developments of numerical methods for fully nonlinear PDEs see Feng, Glowinski, Neilan \cite{FGN13} and Neilan, Salgado, Zhang \cite{NSZ17}. 

\medskip

Periodic HJBI problems of the form \eqref{Intro1} arise naturally as corrector problems in the periodic homogenization of HJBI equations, which is the focus of the second part of this paper. More precisely, we are interested in the numerical approximation of the effective Hamiltonian corresponding to HJBI operators $F:\R^n\times \R^n\times \R^n\times \R^{n\times n}_{\mathrm{sym}}\rightarrow \R$ of the form
\begin{align*}
F(x,y,p,R):=\inf_{\alpha\in\calA}\sup_{\beta\in\calB}\left\{-A^{\alpha\beta}(x,y):R-b^{\alpha\beta}(x,y)\cdot p-f^{\alpha\beta}(x,y)  \right\}
\end{align*}
with sufficiently regular coefficients which are $Y$-periodic in $y\in\R^n$. 

To any fixed point $(x,p,R)\in \R^n\times \R^n\times \R^{n\times n}_{\mathrm{sym}}$ we associate the \textit{approximate correctors} $\{v^{\sigma}(\cdot\,;x,p,R)\}_{\sigma > 0}\subset C(\R^n)$, defined as the unique viscosity solutions to the \textit{cell $\sigma$-problem} (see Alvarez, Bardi \cite{AB10}) for parameters $\sigma>0$, that is, 
\begin{align*}
\left\{\begin{aligned}
\sigma v^{\sigma}(y;x,p,R)+F(x,y,p,R+\nabla_y^2 v^{\sigma}(y;x,p,R))=0\quad\text{for }y\in Y,\\ y\mapsto v^{\sigma}(y;x,p,R)\text{ is $Y$-periodic}.
\end{aligned}\right.
\end{align*}
The operator $F$ is called \textit{ergodic} (in the $y$-variable) at the point $(x,p,R)$ if there exists a constant $H(x,p,R)$ such that
\begin{align*}
-\sigma v^{\sigma}(\cdot\,;x,p,R) \underset{\sigma\searrow 0}{\longrightarrow} H(x,p,R)\quad\text{uniformly},
\end{align*}
and we say $F$ is ergodic if $F$ is ergodic at every point $(x,p,R)$ and call the function
\begin{align*}
H:\R^n\times \R^n\times \R^{n\times n}_{\mathrm{sym}}\rightarrow \R,\qquad (x,p,R)\mapsto H(x,p,R)
\end{align*}
the effective Hamiltonian corresponding to $F$; see Alvarez, Bardi \cite{AB10}. 

The cell $\sigma$-problem is an approximation to the \textit{true cell problem} familiar to the reader coming from periodic homogenization (see Evans \cite{Eva89,Eva92}), that is, for fixed $(x,p,R)$ there exists at most one constant $\mu\in\R$ such that there exists a viscosity solution $v(\cdot\,;x,p,R)\in C(\R^n)$, a \textit{corrector}, to the problem
\begin{align*}
\left\{\begin{aligned}
F(x,y,p,R+\nabla^2 v(y;x,p,R))=\mu \quad\text{for }y\in Y,\\ y\mapsto v(y;x,p,R)\text{ is $Y$-periodic},
\end{aligned}\right.
\end{align*} 
and when such a $\mu$ exists, $F$ is ergodic at $(x,p,R)$ and we have that $H(x,p,R)=\mu$. However, to a given ergodic operator there may be no corrector in general, and we refer to Alvarez, Bardi \cite{AB01,AB03,AB07,AB10}, Alvarez, Bardi, Marchi \cite{ABM07}, and Arisawa, Lions \cite{AL98} for a detailed overview.

The goal of this second part of the paper is the construction of a numerical scheme for the approximation of the effective Hamiltonian to ergodic HJBI operators which is based on discontinuous Galerkin or $C^0$-IP finite element approximations to the approximate correctors; see Section \ref{Chapter 3}.

The literature on numerical effective Hamiltonians to second-order HJB and HJBI operators is quite sparse. For the numerical homogenization of linear equations in nondivergence-form we refer the reader to Capdeboscq, Sprekeler, S\"{u}li \cite{CSS20} (see also Sprekeler, Tran \cite{ST20}). The numerical homogenization of HJB equations via a mixed finite element approximation of the approximate correctors has been proposed and analyzed in Gallistl, Sprekeler, S\"{u}li \cite{GSS21}. A finite difference approach for numerical effective Hamiltonians to HJB operators can be found in Camilli, Marchi \cite{CM09}, and some exact formulas and numerical simulations for effective Hamiltonians to certain types of HJB operators are available in Finlay, Oberman \cite{FO18,FOO18}. 

It seems that there are no finite element schemes for the numerical approximation of effective Hamiltonians to HJBI operators in the current literature. Let us note that there is significantly more work (see e.g., \cite{ACC08,FR08,GLQ18,GO04,LYZ11,OTV09,Qia03,QTY18}) on numerical effective Hamiltonians to first-order Hamilton--Jacobi and Hamilton--Jacobi--Isaacs equations.

\medskip

This paper is organized as follows: Section \ref{Section FE schemes for periodic HJBI} is focused on the DG and $C^0$-IP finite element approximation to the periodic HJBI problem \eqref{Intro1}. After proving existence and uniqueness of a periodic strong solution in Section \ref{Sec wellp}, we discuss discretization and notation aspects in Section \ref{Sec Discretization}. We perform an \textit{a posteriori} analysis independent of the choice of numerical scheme in Section \ref{Section Apo}, which is based on periodic enrichment and a mixed \textit{a posteriori} bound. In Section \ref{Sec a priori analysis}, we perform an \textit{a priori} error analysis for an abstract numerical scheme under natural assumptions, and present a family of numerical schemes in Section \ref{Sec family of num schemes}.

Section \ref{Chapter 3} is focused on the numerical approximation of the effective Hamiltonian to ergodic HJBI operators. We recall the definition of ergodicity and introduce the effective Hamiltonian in Section \ref{Sec3.1}. Thereafter, in Sections \ref{Sec:Appr of cell sigma} and \ref{Sec3.3}, we present the approximation scheme for the effective Hamiltonian based on DG/$C^0$-IP finite element approximations to the cell $\sigma$-problem.

In Section \ref{Chap4}, we present numerical experiments demonstrating the performance of the numerical scheme for a periodic HJBI problem (Section \ref{Sec4.1}) and the approximation of the effective Hamiltonian to an ergodic HJBI operator (Section \ref{Sec 4.2}).

\section{Discontinuous Galerkin and $C^0$-IP FEM for Periodic HJBI Problems}\label{Section FE schemes for periodic HJBI}

\subsection{Setting}\label{Section 2 setting}
Throughout this work, we work in dimension $n\in\{2,3\}$ and write $Y:=(0,1)^n$ to denote the unit cell in $\R^n$. We are interested in Hamilton--Jacobi--Bellman--Isaacs (HJBI) equations posed in a periodic setting, i.e., problems of the form
\begin{align}\label{HJBI}
\left\{\begin{aligned}
F[u]:=\inf_{\alpha\in\calA}\sup_{\beta\in\calB}\left\{ -A^{\alpha\beta}:\nabla^2 u - b^{\alpha\beta}\cdot \nabla u + c^{\alpha\beta} u -f^{\alpha\beta}\right\} = 0 \quad \text{in }Y,\\ u \text{ is $Y$-periodic},
\end{aligned}\right.
\end{align}
with $\calA$ and $\calB$ denoting compact metric spaces, and uniformly continuous functions 
\begin{align*}
A=(a_{ij}):\R^n\times\calA\times\calB\rightarrow \R^{n\times n}_{\mathrm{sym}},\quad b=(b_i):\R^n\times\calA\times\calB\rightarrow \R^{n},\quad c,f:\R^n\times\calA\times\calB\rightarrow \R
\end{align*}
satisfying the assumptions specified below. Here, we use the notation 
\begin{align*}
\fhi^{\alpha\beta}(y):=\fhi(y,\alpha,\beta),\qquad y\in\R^n,\;(\alpha,\beta)\in\calA\times\calB
\end{align*}
for scalar, vector-valued or matrix-valued functions $\fhi\in C(\R^n\times\calA\times\calB;\calR)$ with $\calR\in\{\R,\R^{n},\R^{n\times n}_{\mathrm{sym}}\}$.

We assume that $A^{\alpha\beta},b^{\alpha\beta},c^{\alpha\beta},f^{\alpha\beta}$ are $Y$-periodic in $y\in\R^n$ and that
\begin{align*}
\inf_{\R^n\times\calA\times\calB} c>0.
\end{align*}
We further require $A$ to be uniformly elliptic, i.e.,   
\begin{align}\label{unifm ell}
\exists\,\zeta_1,\zeta_2>0:\quad \zeta_1 \lvert \xi\rvert^2\leq  A(y,\alpha,\beta)\xi \cdot \xi\leq \zeta_2 \lvert \xi\rvert^2\qquad \forall y,\xi\in \R^n,\, (\alpha,\beta)\in \calA\times\calB,
\end{align}
and that the coefficients satisfy the Cordes condition (see \cite{SS14}), i.e., that there holds
\begin{align}\label{Cordes periodic}
\lvert A\rvert^2+\frac{\lvert b\rvert^2}{2\lambda} + \frac{c^2}{\lambda^2}\leq \frac{1}{n+\delta}\left(\mathrm{tr}( A) + \frac{c}{\lambda}  \right)^{2}   
\end{align}
in $\R^n\times \calA\times\calB$ for some constants $\delta\in (0,1)$ and $\lambda>0$ (note $\lvert M\rvert:=\sqrt{M:M}$ for $M\in\R^{n\times n}$).

\subsection{Well-posedness}\label{Sec wellp}

In this section, we show that the periodic HJBI problem \eqref{HJBI} is well-posed in the sense that there exists a unique periodic strong solution, i.e., a unique function $u\in H^2_{\mathrm{per}}(Y)$ satisfying $F[u]=0$ almost everywhere in $Y$. Recall that the space $H^2_{\mathrm{per}}(Y)\subset H^2(Y)$ is defined as the closure of $C^{\infty}_{\mathrm{per}}(Y):=\{\left.v\right\rvert_Y: v\in C^{\infty}(\R^n) \text{ is $Y$-periodic}\}$ with respect to the $H^2$-norm.   

\subsubsection{The renormalized problem}

Let us introduce the function $\gamma=\gamma(y,\alpha,\beta)\in C(\R^n\times\calA\times\calB)$ defined by
\begin{align}\label{gamma first def}
\gamma:= \left(\lvert A\rvert^2+\frac{\lvert b\rvert^2}{2\lambda} + \frac{c^2}{\lambda^2}  \right)^{-1} \left(\mathrm{tr}( A) + \frac{c}{\lambda}\right)
\end{align}
and note that, by the assumptions on the coefficients $A,b,c$ from Section \ref{Section 2 setting}, we have 
\begin{align}\label{positivity gamma}
\inf_{\R^n\times \calA\times\calB}\gamma >0.
\end{align}
We then consider the renormalized HJBI problem
\begin{align}\label{HJBI renormalized}
\left\{\begin{aligned}
F_{\gamma}[u]:=\inf_{\alpha\in\calA}\sup_{\beta\in\calB}\left\{ \gamma^{\alpha\beta}\left(-A^{\alpha\beta}:\nabla^2 u - b^{\alpha\beta}\cdot \nabla u + c^{\alpha\beta} u -f^{\alpha\beta}\right)\right\} = 0 \quad \text{in }Y,\\ u \text{ is $Y$-periodic}.
\end{aligned}\right.
\end{align}
It is easily checked that the renormalized problem \eqref{HJBI renormalized} is equivalent to the original problem \eqref{HJBI} in the sense that they have the same set of periodic strong solutions. More precisely, we can characterize strong solutions to \eqref{HJBI} as follows:

\begin{rmrk}\label{Rk: Characterization of sol}
For $u\in H^2_{\mathrm{per}}(Y)$, the following assertions are equivalent:
\begin{itemize}
\item[(i)] $F[u]=0$ a.e. in $Y$, i.e., $u$ is a periodic strong solution to the HJBI problem \eqref{HJBI}.
\item[(ii)] $F_{\gamma}[u]=0$ a.e. in $Y$, i.e., $u$ is a periodic strong solution to the renormalized  problem \eqref{HJBI renormalized}.
\item[(iii)] There holds
\begin{align*}
\int_{Y} F_{\gamma}[u] L_{\lambda} v = 0\qquad \forall v\in H^2_{\mathrm{per}}(Y),
\end{align*}
where $L_{\lambda}v:= \lambda v -\Delta v$ for functions $v\in H^2_{\mathrm{per}}(Y)$.
\end{itemize}
\end{rmrk}

Indeed, the equivalence (i)$\Leftrightarrow$(ii) follows from \eqref{positivity gamma} and the compactness of the metric spaces $\calA$ and $\calB$ (see also \cite[Lemma 2.2]{KS21}), and (ii)$\Leftrightarrow$(iii) is a consequence of the surjectivity of the linear differential operator
\begin{align*}
L_{\lambda}:H^2_{\mathrm{per}}(Y)\rightarrow L^2(Y),\qquad L_{\lambda}v:= \lambda v -\Delta v.
\end{align*}

\subsubsection{Consequences of the Cordes condition}

We point out a crucial estimate for the nonlinear operator $F_{\gamma}$. This is a direct consequence of the Cordes condition \eqref{Cordes periodic} and can be found in \cite{KS21}. A short proof is provided for demonstrating how the Cordes condition comes into play.

\begin{lmm}\label{Lmm: consequence of Cordes}
Let $\omega\subset Y$ be an open set. For any $u_1,u_2\in H^2(\omega)$, writing $\delta_u:=u_1-u_2$, we have that
\begin{align}\label{Fgamma prop1}
\left\lvert F_{\gamma}[u_1]-F_{\gamma}[u_2]-L_{\lambda} \delta_u\right\rvert \leq \sqrt{1-\delta}\sqrt{\lvert \nabla^2 \delta_u\rvert^2+2\lambda \lvert\nabla \delta_u\rvert^2 +\lambda^2 \delta_u^2 } 
\end{align}
almost everywhere in $\omega$. 
\end{lmm}
\begin{proof}
Let $u_1,u_2\in H^2(\omega)$ and set $\delta_u:=u_1-u_2$. Note that for any bounded sets $\{x^{\alpha\beta}\}_{(\alpha,\beta)\in\calA\times\calB}\subset\R$ and $\{y^{\alpha\beta}\}_{(\alpha,\beta)\in\calA\times\calB}\subset\R$ we have that 
\begin{align*}
\left\lvert\inf_{\alpha\in \calA}\sup_{\beta\in\calB} x^{\alpha\beta}-\inf_{\alpha\in\calA}\sup_{\beta\in\calB} y^{\alpha\beta}\right\rvert \leq \sup_{(\alpha,\beta)\in\calA\times \calB}\lvert x^{\alpha\beta}-y^{\alpha\beta}\rvert.
\end{align*}
This yields
\begin{align*}
\lvert F_{\gamma}[u_1]-&F_{\gamma}[u_2]-L_{\lambda} \delta_u\rvert^2  \leq \sup_{(\alpha,\beta)\in\calA\times\calB} \left\lvert \gamma^{\alpha\beta}\left(-A^{\alpha\beta}:\nabla^2 \delta_u-b^{\alpha\beta}\cdot \nabla \delta_u+c^{\alpha\beta}\delta_u\right)+\Delta \delta_u-\lambda \delta_u\right\rvert^2\\ &\leq \sup_{(\alpha,\beta)\in\calA\times\calB} \left\{ \left\lvert -\gamma^{\alpha\beta}A^{\alpha\beta}+I\right\rvert^2+\frac{\lvert \gamma^{\alpha\beta}b^{\alpha\beta}\rvert^2}{2\lambda}+\frac{\lvert \gamma^{\alpha\beta}c^{\alpha\beta}-\lambda\rvert^2 }{\lambda^2}\right\}\left(\lvert \nabla^2\delta_u\rvert^2 + 2\lambda\left\lvert \nabla \delta_u\right\rvert^2 +\lambda^2 \delta_u^2   \right)\\ &= \sup_{(\alpha,\beta)\in\calA\times\calB} \left\{ n+1-\frac{\left(\mathrm{tr}(A^{\alpha\beta})+\frac{c^{\alpha\beta}}{\lambda}\right)^2}{\lvert A^{\alpha\beta}\rvert^2+\frac{\lvert b^{\alpha\beta}\rvert^2}{2\lambda}+\frac{\lvert c^{\alpha\beta}\rvert^2}{\lambda^2}}\right\}\left(\lvert \nabla^2\delta_u\rvert^2 + 2\lambda\left\lvert \nabla \delta_u\right\rvert^2 +\lambda^2 \delta_u^2   \right)\\
&\leq (1-\delta)\left(\lvert \nabla^2\delta_u\rvert^2 + 2\lambda\left\lvert \nabla \delta_u\right\rvert^2 +\lambda^2 \delta_u^2   \right)
\end{align*}
almost everywhere in $\omega$, where we have used the Cauchy--Schwarz inequality, simple calculation and the Cordes condition \eqref{Cordes periodic}.
\end{proof}

Observe that by the triangle and Cauchy-Schwarz inequalities, we can eliminate the term $L_{\lambda}\delta_u$ from the left-hand side of \eqref{Fgamma prop1}. We thus find that, in the situation of Lemma \ref{Lmm: consequence of Cordes}, we have the Lipschitz-type estimate 
\begin{align}\label{Fgamma prop}
\left\lvert F_{\gamma}[u_1]-F_{\gamma}[u_2]\right\rvert \leq \left(\sqrt{1-\delta}+\sqrt{n+1}\right)\sqrt{\lvert \nabla^2 \delta_u\rvert^2+2\lambda \lvert\nabla \delta_u\rvert^2 +\lambda^2 \delta_u^2 }
\end{align} 
almost everywhere in $\omega$.

\subsubsection{Existence and uniqueness of solutions}

We are now in a position to prove the existence and uniqueness of periodic strong solutions to the HJBI problem \eqref{HJBI}. In view of Remark \ref{Rk: Characterization of sol}, let us define
\begin{align*}
B:H^2_{\mathrm{per}}(Y)\times H^2_{\mathrm{per}}(Y)\rightarrow\R,\qquad B(u,v):= \int_{Y} F_{\gamma}[u] L_{\lambda} v.
\end{align*}

We can now proceed as in \cite{KS21} in showing that the Browder--Minty theorem applies and we obtain the following theorem:

\begin{thrm}[Well-posedness]\label{thm well-pos}
In the situation of Section \ref{Section 2 setting}, there exists a unique periodic strong solution $u\in H^2_{\mathrm{per}}(Y)$ to the HJBI problem \eqref{HJBI}. 
\end{thrm}
\begin{proof}
Note that it is enough to show that $B$ satisfies the Lipschitz property
\begin{align}\label{B Lipschitz}
\lvert B(u_1,v)-B(u_2,v)\rvert\lesssim \|u_1-u_2\|_{H^2(Y)}\|v\|_{H^2(Y)}\qquad \forall u_1,u_2,v\in H^2_{\mathrm{per}}(Y),
\end{align}
and strong monotonicity, i.e.,
\begin{align}\label{B monotone}
\|u_1-u_2\|_{H^2(Y)}^2\lesssim B(u_1,u_1-u_2)-B(u_2,u_1-u_2)\qquad \forall u_1,u_2\in H^2_{\mathrm{per}}(Y). 
\end{align}
The Browder--Minty theorem then yields that there exists a unique $u\in H^2_{\mathrm{per}}(Y)$ such that 
\begin{align*}
B(u,v)=0\qquad \forall v\in H^2_{\mathrm{per}}(Y),
\end{align*}
which proves the theorem in view of Remark \ref{Rk: Characterization of sol}. 

Before we show \eqref{B Lipschitz} and \eqref{B monotone}, let us note that integration by parts and a density argument yields $\|\Delta v\|_{L^2(Y)}=\|\nabla^2 v\|_{L^2(Y)}$ for any $v\in H^2_{\mathrm{per}}(Y)$, and hence, using integration by parts again, we have
\begin{align}\label{Llambda norm}
\|L_{\lambda}v\|_{L^2(Y)}^2=\|\nabla^2 v\|_{L^2(Y)}^2 + 2\lambda \|\nabla v\|_{L^2(Y)}^2 +\lambda^2 \| v\|_{L^2(Y)}^2 \geq C_\lambda \|v\|_{H^2(Y)}^2\qquad \forall v\in H^2_{\mathrm{per}}(Y).
\end{align}

The Lipschitz property \eqref{B Lipschitz} now immediately follows from \eqref{Fgamma prop} and it remains to show strong monotonicity. To this end, let $u_1,u_2\in H^2_{\mathrm{per}}(Y)$ and write $\delta_u:=u_1-u_2$. Using Lemma \ref{Lmm: consequence of Cordes}, we find
\begin{align*}
B(u_1,\delta_u)-B(u_2,\delta_u)=\|L_{\lambda}\delta_u\|_{L^2(Y)}^2+\int_Y \left(F_{\gamma}[u_1]-F_{\gamma}[u_2]-L_{\lambda}\delta_u\right)L_{\lambda}\delta_u\geq (1-\sqrt{1-\delta})\|L_{\lambda}\delta_u\|_{L^2(Y)}^2
\end{align*}
and hence, by \eqref{Llambda norm}, there holds \eqref{B monotone} and the claim is proved.
\end{proof}

\begin{rmrk}\label{Rk: bound on soln}
For the unique periodic strong solution $u\in H^2_{\mathrm{per}}(Y)$ to the HJBI problem \eqref{HJBI}, we have the bound
\begin{align*}
\|L_{\lambda}u\|_{L^2(Y)}=\sqrt{\|\nabla^2 u\|_{L^2(Y)}^2 + 2\lambda \|\nabla u\|_{L^2(Y)}^2 +\lambda^2 \| u\|_{L^2(Y)}^2} \leq  \frac{\|F_{\gamma}[0]\|_{L^2(Y)}}{1-\sqrt{1-\delta}}.
\end{align*}
\end{rmrk}
\begin{proof}
Note that we have already obtained the first equality (see \eqref{Llambda norm}). We use Lemma \ref{Lmm: consequence of Cordes} and the solution property $F_{\gamma}[u]=0$ to find
\begin{align*}
(1-\sqrt{1-\delta})\|L_{\lambda}u\|_{L^2(Y)}^2 \leq \int_Y (F_{\gamma}[u]-F_{\gamma}[0])L_{\lambda}u = -\int_Y F_{\gamma}[0] L_{\lambda}u.
\end{align*}
We conclude the proof by using H\"{o}lder's inequality to obtain
\begin{align*}
\|L_{\lambda}u\|_{L^2(Y)}^2 \leq \frac{1}{1-\sqrt{1-\delta}}\left\lvert\int_Y F_{\gamma}[0] L_{\lambda}u \right\rvert\leq \frac{\|F_{\gamma}[0]\|_{L^2(Y)}}{1-\sqrt{1-\delta}}\|L_{\lambda}u\|_{L^2(Y)},
\end{align*}
which yields the desired bound.
\end{proof}

\subsection{Discretization}\label{Sec Discretization}

This section is devoted to discretization aspects. We introduce DG and $C^0$-IP finite element spaces $V^0_{\calT}$ and $V^1_{\calT}$ for an appropriate partition $\calT$ of the computational domain, and define jump and average operators.

\subsubsection{The partition $\calT$}

We consider a finite conforming partition $\calT$ of the closed unit cell $\bar{Y}$ consisting of closed simplices that can be periodically extended in a $Y$-periodic fashion to $\R^n$, i.e., we require the discretization to be consistent with the identification of opposite faces by periodicity. We introduce the following mathematical objects associated with the partition $\calT$:

\begin{itemize}
\item[(i)] Set of faces $\calF$ and associated unit normal $n_F$: \\We let $\calF:= \calF^{\mathrm{I}}\cup \calF^{\mathrm{BP}}$ denote the set of $(n-1)$-dimensional faces, where $\calF^{\mathrm{I}}$ is the set of all interior faces of $\calT$, and $\calF^{\mathrm{BP}}$ the set of all boundary face-pairs of $\calT$, i.e., the boundary faces upon a periodic identification of opposite faces. For each face $F\in\calF$, we associate a fixed choice of unit normal $n_F$, where we often only write $n$ for simplicity; see Figure \ref{fig: T}. 
\item[(ii)] Shape-regularity parameter $\theta_{\calT}$ and mesh-size function $h_{\calT}$: \\We let $\theta_T:=\max\{\rho_K^{-1} \mathrm{diam}(K):K\in\calT\}$ with $\rho_K$ being the diameter of the largest ball that can be inscribed in the element $K\in\calT$. We further introduce $h_{\calT}:\bar{Y}\rightarrow\R$ defined via $\left.h_{\calT}\right\rvert_{\mathrm{int}(K)}:=h_K:=\left(\calL^n(K)\right)^{\frac{1}{n}}$ for all $K\in\calT$ and  $\left.h_{\calT}\right\rvert_{F}:=h_F:=\left(\calH^{n-1}(F)\right)^{\frac{1}{n-1}}$ for all $F\in\calF$. 
\end{itemize} 

Let us note that the concept of boundary face-pairs was introduced in \cite{Vem07} in the context of discontinuous Galerkin methods for linear elliptic periodic boundary value problems. 

\subsubsection{Finite element spaces $V_{\calT}^s$}

For fixed $\bar{p}\geq 2$, we define the discontinuous Galerkin finite element space $V_{\calT}^0$ and the $C^0$-IP finite element space $V_{\calT}^1$ by
\begin{align*}
V_{\calT}^0:=\left\{v_{\calT}\in L^2(Y):\left.v_{\calT}\right\rvert_K \in \mathbb{P}_{\bar{p}}\;\forall K\in\calT \right\}\quad\text{and}\quad V_{\calT}^1:=V_{\calT}^0\cap H^1_{\mathrm{per}}(Y),
\end{align*}
where $\mathbb{P}_{\bar{p}}$ denotes the space of polynomials of degree at most $\bar{p}$. 

Let us make some comments about the derivatives of functions in the finite element spaces. For a function $v\in V^0_{\calT}$, we define $\nabla v\in L^1(Y;\R^n)$ to be the piecewise gradient and $\nabla^2 v\in L^1(Y;\R^{n\times n})$ to be the piecewise Hessian over the elements of the partition. We then define $\Delta v:=\mathrm{tr}(\nabla^2v)\in L^1(Y)$. 

We equip the finite element spaces $V_{\calT}^s$, $s\in\{0,1\}$, with the norm
\begin{align*}
\|v_{\calT}\|_{\calT,\lambda}^2:= \int_Y \left(\lvert \nabla^2 v_{\calT}\rvert^2 +2\lambda\lvert \nabla v_{\calT}\rvert^2+ \lambda^2 v_{\calT}^2\right) + \lvert v_{\calT}\rvert_{J,\calT}^2,\quad \lvert v_{\calT}\rvert_{J,\calT}^2:= \int_{\calF} \left( h_{\calT}^{-1} \lvert \llbracket\nabla v_{\calT}\rrbracket\rvert^2 + h_{\calT}^{-3}\lvert \llbracket v_{\calT}\rrbracket\rvert^2\right)
\end{align*}
for functions $v_{\calT}\in V^s_{\calT}$. In order to simplify the presentation, throughout this work we write $\int_{\calE}:=\sum_{K\in \calE}\int_K$ for collections $\calE\subset \calT$ of elements and $\int_{\calG}:=\sum_{F\in \calG}\int_F$ for collections $\calG\subset\calF$ of faces. The jump operator $\llbracket\cdot\rrbracket$ is defined in the following paragraph.

\subsubsection{Jump and average operators}

For elements $K\in\calT$, we write $\tau_{\partial K}:\mathrm{BV}(K)\rightarrow L^1(\partial K)$ to denote the trace operator. Further, for $v\in \mathrm{BV}(Y)$ we define $\tau_{\partial K}v:=\tau_{\partial K} (\left.v\right\rvert_K)$ for elements $K\in\calT$. We then introduce the jump $\llbracket v\rrbracket_F$ and the average $\{v\}_F$ of a function $v\in \mathrm{BV}(Y)$ over a face $F=\partial K\cap \partial K'\in\calF$ shared by the elements $K,K'\in\calT$ by
\begin{align*}
\llbracket v\rrbracket_F&:=\left.\tau_{\partial K} v\right\rvert_F - \left.\tau_{\partial K'} v\right\rvert_F \in L^1(F),\\ \{v\}_F&:=\frac{\left.\tau_{\partial K} v\right\rvert_F + \left.\tau_{\partial K'} v\right\rvert_F}{2}\in L^1(F),
\end{align*} 
where $K,K'$ are labeled such that the unit normal $n_F$ is the outward normal to $K$ on the face $F$; see Figure \ref{fig: T}.  
To simplify the presentation, we will often simply write  $\llbracket \cdot\rrbracket$ and $\{ \cdot\}$, and drop the subscript.

\begin{figure}
\centering
\includegraphics[scale=0.66]{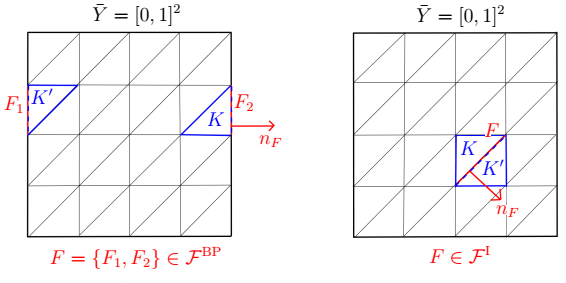}
\caption{Illustration of a boundary face-pair $F\in \calF^{\mathrm{BP}}$ (left) and an interior face $F\in \calF^{\mathrm{I}}$ (right) in dimension $n=2$.}
\label{fig: T}
\end{figure}

\subsection{\textit{A posteriori} analysis}\label{Section Apo}

Let $u\in H^2_{\mathrm{per}}(Y)$ denote the unique solution to the HJBI problem \eqref{HJBI} and let $v_{\calT}\in V^0_{\calT}$ be arbitrary. The goal of this section is to estimate the $\|\cdot\|_{\calT,\lambda}$-distance between $u$ and $v_{\calT}$, i.e., 
\begin{align*}
\|u-v_{\calT}\|_{\calT,\lambda}^2=\int_Y \left(\lvert \nabla^2 (u-v_{\calT})\rvert^2 +2\lambda\lvert \nabla (u-v_{\calT})\rvert^2+ \lambda^2 (u-v_{\calT})^2\right) + \lvert u-v_{\calT}\rvert_{J,\calT}^2,
\end{align*}
in terms of a computable quantity not depending on the solution $u$. We start by introducing periodic enrichment operators which are an important tool in establishing the \textit{a posteriori} bound.

\subsubsection{Periodic enrichment}

We let $\calZ$ be the set of points in $\bar{Y}$ corresponding to the Lagrange degrees of freedom for the function space $V^1_{\calT}=V^0_{\calT}\cap H^1_{\mathrm{per}}(Y)$, where boundary nodes on $\partial Y$ are identified with all their $Y$-periodic counterparts. For $z\in\calZ$, we then define the \textit{periodic neighborhood} $N(z)\subset \calT$ to be the set of all elements $K\in \calT$ that contain $z$ or any periodically identical point to $z$; see Figure \ref{fig: N(z)}.

\begin{figure}
\centering
\includegraphics[scale=0.6]{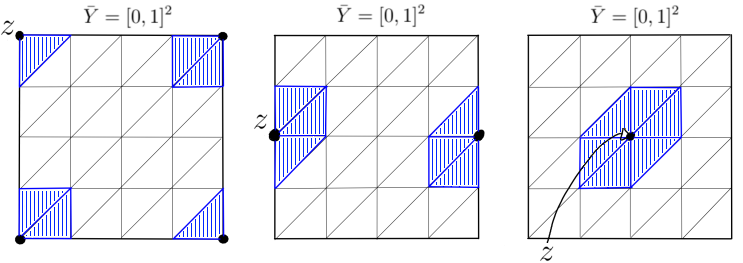}
\caption{Illustration of the periodic neighborhood $N(z)\subset \calT$ in dimension $n=2$. \textit{Left}: $z\in\calZ\cap \partial Y$ corner point, \textit{middle}: $z\in\calZ\cap \partial Y$ non-corner boundary point, \textit{right}: $z\in\calZ\cap Y$ interior point.}
\label{fig: N(z)}
\end{figure}

Let us introduce an operator
\begin{align*}
E_1:V^0_{\calT}\rightarrow V^0_{\calT}\cap H^1_{\mathrm{per}}(Y),
\end{align*}
which we call the $H^1_{\mathrm{per}}$-enrichment operator, defined through averaging of the function values in periodic neighborhoods of points in $\calZ$. That is, for $v_{\calT}\in V^0_{\calT}$, we define the function $E_1 v_{\calT}\in V^1_{\calT}$ by prescribing
\begin{align*}
E_1v_{\calT}(z):=\frac{1}{\lvert N(z)\rvert} \sum_{K\in N(z)}\left. v_{\calT}\right\rvert_{K}(z)
\end{align*}
at points $z\in\calZ$. Denoting the collection of interior faces and boundary face-pairs neighboring an element $K\in\calT$ by $\calF_K:=\{F\in\calF:F\cap K\neq \emptyset\}$, we then have the bound
\begin{align}\label{E1 bound}
\sum_{m=0}^2\int_K h_{\calT}^{2m-4}  \left\lvert \nabla^m(v_{\calT}-E_1 v_{\calT})\right\rvert^2\lesssim \int_{\calF_K}h_{\calT}^{-3} \lvert\llbracket v_{\calT}\rrbracket\rvert^2\qquad \forall K\in\calT
\end{align}
for all $v_{\calT}\in V^0_{\calT}$, where the constant absorbed in $\lesssim$ only depends on $n,\theta_{\calT}$ and $\bar{p}$. This bound follows from the arguments in \cite{KP03}.

Let us also discuss the periodic enrichment of vector fields. To this end, we define the space containing potential gradients of functions in the finite element spaces by
\begin{align*}
W_{\calT}:=\{v_{\calT}\in L^2(Y;\R^n):\left. v_{\calT}\right\rvert_K\in\mathbb{P}_{\bar{p}-1}^n\;\forall K\in\calT\}.
\end{align*}
Indeed, observe that $\nabla v_{\calT}\in W_{\calT}$ for any $v_{\calT}\in V^s_{\calT}$, $s\in\{0,1\}$. Analogously to $E_1$, we can then construct a linear operator 
\begin{align*}
E_1^g:W_{\calT}\rightarrow W_{\calT}\cap H^1_{\mathrm{per}}(Y;\R^n)
\end{align*}
satisfying
\begin{align}\label{ET bound}
\int_K \left( \left\lvert \nabla(w_{\calT}-E_1^g w_{\calT})\right\rvert^2+ h_{\calT}^{-2}\left\lvert w_{\calT}-E_1^g w_{\calT}\right\rvert^2\right)\lesssim \int_{\calF_K}h_{\calT}^{-1} \lvert\llbracket w_{\calT}\rrbracket\rvert^2\qquad \forall K\in\calT
\end{align}
for all $w_{\calT}\in W_{\calT}$, where the constant absorbed in $\lesssim$ only depends on $n,\theta_{\calT}$ and $\bar{p}$. With the enrichment operators at hand we can proceed with the \textit{a posteriori} analysis, independent of the choice of the numerical scheme. 

\subsubsection{The \textit{a posteriori} bound}\label{Sec: A posteriori} 

It will be useful to introduce some notation from the mixed finite element theory developed in \cite{GSS21}. Let us consider the function space
\begin{align*}
X:=W_{\mathrm{per}}(Y;\R^n)\times H^1_{\mathrm{per}}(Y),
\end{align*}
which we equip with the $\vertiii{\cdot}_{\lambda}$-norm given by
\begin{align*}
\vertiii{(w',u')}_{\lambda}^2:= \|\nabla w'\|_{L^2(Y)}^2 + 2\lambda\|\nabla u'\|_{L^2(Y)}^2 +\lambda^2\|u'\|_{L^2(Y)}^2,\qquad (w',u')\in X.
\end{align*}
We recall that the spaces $W_{\mathrm{per}}(Y)\subset H^1_{\mathrm{per}}(Y)$ and $W_{\mathrm{per}}(Y;\R^n)\subset H^1_{\mathrm{per}}(Y;\R^n)$ are defined as
\begin{align*}
W_{\mathrm{per}}(Y):=\left\{ v\in H^1_{\mathrm{per}}(Y):\int_Y v = 0\right\},\qquad W_{\mathrm{per}}(Y;\R^n):= \left(W_{\mathrm{per}}(Y)\right)^n.
\end{align*}
We further define the mixed analogue $F_{\gamma}^{M}$ to the nonlinear operator $F_{\gamma}$ by
\begin{align*}
F_{\gamma}^{M}[(w',u')]:=\inf_{\alpha\in\calA}\sup_{\beta\in\calB}\left\{ \gamma^{\alpha\beta}\left(-A^{\alpha\beta}:\nabla w' - b^{\alpha\beta}\cdot \nabla u' + c^{\alpha\beta} u' -f^{\alpha\beta}\right)\right\}
\end{align*}
for pairs $(w',u')\in X$, and observe that the solution $u\in H^2_{\mathrm{per}}(Y)$ to \eqref{HJBI} satisfies
\begin{align}\label{Fgamma mixed equ}
F_{\gamma}^M[(\nabla u,u)]=F_{\gamma}[u]=0\quad\text{a.e. in }Y.
\end{align}

We can use the arguments from \cite{GSS21} to prove an \textit{a posteriori} bound on the $\vertiii{\cdot}_{\lambda}$-distance between the solution pair $(\nabla u,u)$ and an arbitrary pair $(w',u')\in X$.
\begin{lmm}\label{Lemma mixed}
Let $u\in H^2_{\mathrm{per}}(Y)$ denote the unique solution to the HJBI problem \eqref{HJBI}. Then we have
\begin{align*}
\vertiii{(\nabla u-w',u-u')}_{\lambda}^2 \lesssim \|F_{\gamma}^M[(w',u')]\|_{L^2(Y)}^2+ \|\mathrm{rot}(w')\|_{L^2(Y)}^2 +\|\nabla u' -w'\|_{L^2(Y)}^2\quad \forall\, (w',u')\in X
\end{align*}
with the constant absorbed in $\lesssim$ only depending on the Cordes parameters $\delta,\lambda$.
\end{lmm}
\begin{proof}
We define the semilinear form $a^M:X\times X\rightarrow \R$ by
\begin{align*}
&a^M\left((w_1,u_1),(w_2,u_2) \right)\\ &\quad:=\int_Y F_{\gamma}^M[(w_1,u_1)](\lambda u_2-\nabla\cdot w_2) + \sigma_1 \int_Y \mathrm{rot}(w_1)\cdot \mathrm{rot}(w_2)+\sigma_2\int_Y (\nabla u_1-w_1)\cdot (\nabla u_2-w_2)
\end{align*}
with $\sigma_1,\sigma_2>0$ given by 
\begin{align*}
\sigma_1:= 1-\frac{1}{2}\sqrt{1-\delta},\qquad \sigma_2:=\frac{\lambda}{2}(1-\sqrt{1-\delta})+\frac{\lambda}{4}(1-\sqrt{1-\delta})^{-1}.
\end{align*}
A straightforward adaptation of the proof of \cite[Lemma 2.3]{GSS21} yields the monotonicity estimate
\begin{align*}
C_{\delta}\vertiii{(w_1-w_2,u_1-u_2)}_{\lambda}^2 \leq a^M\left((w_1,u_1),(w_1-w_2,u_1-u_2) \right)-a^M\left((w_2,u_2),(w_1-w_2,u_1-u_2) \right)
\end{align*}
for all $(w_1,u_1),(w_2,u_2)\in X$, where $C_{\delta}>0$ is a constant only depending on $\delta$. In particular, in view of \eqref{Fgamma mixed equ}, we find that
\begin{align}\label{aM ineq}
C_{\delta}\vertiii{(\nabla u-w',u-u')}_{\lambda}^2\leq -a^M\left((w',u'),(\nabla u-w',u-u')\right)\qquad\forall\,(w',u')\in X.
\end{align}

Let $(w',u')\in X$ be arbitrary and write $(\delta_w,\delta_u):=(\nabla u-w',u-u')$. Using the Cauchy--Schwarz and Young inequalities to bound the right-hand side of \eqref{aM ineq}, we have
\begin{align*}
C_{\delta} &\vertiii{(\delta_w,\delta_u)}_{\lambda}^2 \leq \left\lvert a^M\left((w',u'),(\delta_w,\delta_u)\right)\right\rvert \\ &\hspace{1cm}\leq \frac{1}{C_{\delta}}\|F_{\gamma}^M[(w',u')]\|_{L^2(Y)}^2+ \frac{C_{\delta}}{4}\|\lambda\delta_u-\nabla\cdot \delta_w\|_{L^2(Y)}^2+\sigma_1\|\mathrm{rot}(w')\|_{L^2(Y)}^2 +\sigma_2\|\nabla u' -w'\|_{L^2(Y)}^2
\end{align*}
and we can conclude that
\begin{align*}
\vertiii{(\delta_w,\delta_u)}_{\lambda}^2\leq \frac{1}{C_{\delta}^2}\|F_{\gamma}^M[(w',u')]\|_{L^2(Y)}^2+ \frac{1}{2}\vertiii{(\delta_w,\delta_u)}_{\lambda}^2+\frac{\sigma_1}{C_{\delta}}\|\mathrm{rot}(w')\|_{L^2(Y)}^2 +\frac{\sigma_2}{C_{\delta}}\|\nabla u' -w'\|_{L^2(Y)}^2
\end{align*}
upon noting $\|\nabla\cdot \delta_w\|_{L^2(Y)}\leq \|\nabla \delta_w\|_{L^2(Y)}$ as $\delta_w\in H^1_{\mathrm{per}}(Y;\R^n)$; see \cite{GSS21}. Finally, absorbing the term $\frac{1}{2}\vertiii{(\delta_w,\delta_u)}_{\lambda}^2$ into the left-hand side of the above inequality, we obtain the desired estimate.
\end{proof}

We can use Lemma \ref{Lemma mixed} and the $H^1_{\mathrm{per}}$-enrichment operators to prove the following \textit{a posteriori} error bound:

\begin{thrm}[\textit{a posteriori} error bound]\label{thm: a posteriori}
Let $u\in H^2_{\mathrm{per}}(Y)$ denote the unique solution to the HJBI problem \eqref{HJBI}. Then there holds
\begin{align*}
\|u-v_{\calT}\|_{\calT,\lambda}^2\lesssim \int_Y \left\lvert F_{\gamma}[v_{\calT}]\right\rvert^2+\lvert v_{\calT}\rvert_{J,\calT}^2\qquad \forall v_{\calT}\in V^0_{\calT}
\end{align*}
with the constant absorbed in $\lesssim$ only depending on $n,\theta_{\calT},\bar{p}$ and the Cordes parameters $\delta,\lambda$.
\end{thrm}
\begin{proof}
Let $v_{\calT}\in V^0_{\calT}$ be arbitrary and set 
\begin{align*}
v&:=E_1 v_{\calT}\in V^0_{\calT}\cap H^1_{\mathrm{per}}(Y),\\ w&:=E_1^g(\nabla v_{\calT})-\int_Y E_1^g(\nabla v_{\calT})\in W_{\calT}\cap W_{\mathrm{per}}(Y;\R^n).
\end{align*}
By the triangle inequality, we have 
\begin{align*}
\|u-v_{\calT}\|_{\calT,\lambda}^2  &\lesssim \vertiii{(\nabla u-w,u-v)}_{\lambda}^2 \\ &\quad + \int_Y \left(\lvert \nabla(w-\nabla v_{\calT})\rvert^2 +2\lambda\lvert \nabla (v-v_{\calT})\rvert^2+ \lambda^2 (v-v_{\calT})^2\right)+\lvert v_{\calT}\rvert_{J,\calT}^2,
\end{align*}
which we can further bound, using the properties of the enrichment operators \eqref{E1 bound} and \eqref{ET bound}, to obtain that 
\begin{align*}
\|u-v_{\calT}\|_{\calT,\lambda}^2 \lesssim  \vertiii{(\nabla u-w,u-v)}_{\lambda}^2+\lvert v_{\calT}\rvert_{J,\calT}^2.
\end{align*}
We can apply Lemma \ref{Lemma mixed} to find
\begin{align}\label{three terms}
\|u-v_{\calT}\|_{\calT,\lambda}^2 \lesssim \|F_{\gamma}^M[(w,v)]\|_{L^2(Y)}^2+ \|\mathrm{rot}(w)\|_{L^2(Y)}^2 +\|\nabla v -w\|_{L^2(Y)}^2 + \lvert v_{\calT}\rvert_{J,\calT}^2.
\end{align}
Note that, using the triangle and H\"{o}lder inequalities, and the enrichment bounds \eqref{E1 bound} and \eqref{ET bound}, we have
\begin{align*}
\|\mathrm{rot}(w)\|_{L^2(Y)}^2 = \int_Y \lvert \mathrm{rot}(w-\nabla v_{\calT})  \rvert^2 \lesssim \lvert v_{\calT}\rvert_{J,\calT}^2
\end{align*}
for the second term on the right-hand side of \eqref{three terms}, and
\begin{align*}
\|\nabla v -w\|_{L^2(Y)}^2 &\lesssim \left\|\nabla v - E_1^g(\nabla v_{\calT})-\int_Y \left(\nabla v -E_1^g(\nabla v_{\calT})\right)  \right\|_{L^2(Y)}^2    \\ &\lesssim  \left\|\nabla v - E_1^g(\nabla v_{\calT}) \right\|_{L^2(Y)}^2 \\ &\lesssim    \int_Y \lvert \nabla(v-v_{\calT})\rvert^2 + \int_Y \lvert \nabla v_{\calT}-E_1^g(\nabla v_{\calT})\rvert^2 \\ &\lesssim \lvert v_{\calT}\rvert_{J,\calT}^2 
\end{align*}
for the third term on the right-hand side of \eqref{three terms} (note that $\int_Y \nabla v = 0$ for $v\in H^1_{\mathrm{per}}(Y)$). Finally, for the first term on the right-hand side of \eqref{three terms}, we successively use the triangle inequality together with $F_{\gamma}[v_{\calT}]=F_{\gamma}^M[(\nabla v_{\calT},v_{\calT})]$, a Lipschitz property of $F_{\gamma}^M$ which is shown analogously to \eqref{Fgamma prop}, and the enrichment bounds \eqref{E1 bound} and \eqref{ET bound} to obtain
\begin{align*}
\|F_{\gamma}^M[(w,v)]\|_{L^2(Y)}^2&\lesssim \int_Y \left\lvert F_{\gamma}[v_{\calT}]\right\rvert^2+\int_Y \left\lvert F_{\gamma}^M[(w,v)]-F_{\gamma}^M[(\nabla v_{\calT},v_{\calT})]\right\rvert^2\\
&\lesssim  \int_Y \left\lvert F_{\gamma}[v_{\calT}]\right\rvert^2+\int_Y \left(\lvert \nabla(w-\nabla v_{\calT})\rvert^2+2\lambda \lvert \nabla(v- v_{\calT})\rvert^2+\lambda^2  \lvert v- v_{\calT}\rvert^2\right)\\&\lesssim  \int_Y \left\lvert F_{\gamma}[v_{\calT}]\right\rvert^2+\lvert v_{\calT}\rvert_{J,\calT}^2 .
\end{align*}
Altogether, in view of \eqref{three terms}, we have proved the desired estimate.
\end{proof}

This concludes the \textit{a posteriori} analysis and we proceed with an abstract \textit{a priori} analysis for a wide class of numerical schemes in the next section.

\subsection{Numerical scheme and \textit{a priori} analysis}\label{Sec a priori analysis}

Let us consider an abstract numerical scheme written in the following form: For chosen $s\in\{0,1\}$, find a function $u_{\calT}\in V^s_{\calT}$ satisfying
\begin{align}\label{Abstract numerical scheme}
a_{\calT}(u_{\calT},v_{\calT})=0\qquad \forall v_{\calT}\in V_{\calT}^s.
\end{align} 

\subsubsection{Abstract \textit{a priori} analysis}

Here, we assume that the nonlinear form $a_{\calT}:V_{\calT}^s\times V_{\calT}^s\rightarrow\R$ satisfies the assumptions listed below:
\begin{itemize}
\item[(A1)] Linearity in second argument: $a_{\calT}(w_{\calT},\cdot\,):V_{\calT}^s\rightarrow\R$ is linear for any fixed $w_{\calT}\in V_{\calT}^s$.
\item[(A2)] Strong monotonicity: There exists a constant $C_M>0$ such that 
\begin{align*}
\|w_{\calT}-v_{\calT}\|^2_{\calT,\lambda}\leq C_M\left(a_{\calT}(w_{\calT},w_{\calT}-v_{\calT})-a_{\calT}(v_{\calT},w_{\calT}-v_{\calT})  \right)\qquad \forall w_{\calT},v_{\calT}\in V^s_{\calT}.
\end{align*}
\item[(A3)] Lipschitz continuity: There exists a constant $C_L>0$ such that 
\begin{align*}
\left\lvert a_{\calT}(w_{\calT},v_{\calT})-a_{\calT}(w_{\calT}',v_{\calT})\right\rvert\leq C_L \|w_{\calT}-w'_{\calT}\|_{\calT,\lambda}\|v_{\calT}\|_{\calT,\lambda}\qquad \forall w_{\calT},w'_{\calT},v_{\calT}\in V^s_{\calT}.
\end{align*}
\item[(A4)] Discrete consistency: There exists a linear operator $L_{\calT}:V^s_{\calT}\rightarrow L^2(Y)$ such that, for some constant $C_1>0$, we have 
\begin{align*}
\|L_{\calT} v_{\calT}\|_{L^2(Y)}\leq C_1 \|v_{\calT}\|_{\calT,\lambda}\qquad \forall v_{\calT}\in V^s_{\calT},
\end{align*}
and, for some constant $C_2>0$, we have
\begin{align*}
\left\lvert a(w_{\calT},v_{\calT})-\int_Y F_{\gamma}[w_{\calT}]L_{\calT}v_{\calT} \right\rvert\leq C_2\lvert w_{\calT}\rvert_{J,\calT}\|v_{\calT}\|_{\calT,\lambda}\qquad \forall w_{\calT},v_{\calT}\in V^s_{\calT}.
\end{align*}   
\end{itemize}

Observe that the assumptions (A1)--(A4) guarantee well-posedness of the numerical scheme, that is, there exists a unique solution $u_{\calT}\in V^s_{\calT}$ satisfying \eqref{Abstract numerical scheme}. We can show an \textit{a priori} bound in this general setting similarly to \cite{KS21}. 

\begin{thrm}[\textit{a priori} error bound]\label{Tm: a priori bound}
For chosen $s\in\{0,1\}$, let $a_{\calT}:V^s_{\calT}\times V^s_{\calT}\rightarrow \R$ be a nonlinear form satisfying the assumptions (A1)--(A4). Further, let $u\in H^2_{\mathrm{per}}(Y)$ denote the unique solution to the HJBI problem \eqref{HJBI}. Then, there exists a unique solution $u_{\calT}\in V^s_{\calT}$ to \eqref{Abstract numerical scheme} and we have the near-best approximation bound
\begin{align}\label{a priori bound}
\|u-u_{\calT}\|_{\calT,\lambda}\leq C_e \inf_{v_{\calT}\in V^s_{\calT}}\|u-v_{\calT}\|_{\calT,\lambda},
\end{align}
where the constant $C_e>0$ is given by
\begin{align}\label{constant Ce}
C_e:= 1+C_M\left(C_1\left(\sqrt{1-\delta}+\sqrt{n+1}\right)+C_2\right).
\end{align}
\end{thrm}

\begin{proof}
As we have already noted, the existence and uniqueness of a solution $u_{\calT}\in V^s_{\calT}$ to \eqref{Abstract numerical scheme} follows from the assumptions on the nonlinear form $a_{\calT}$, and it only remains to show the near-best approximation bound \eqref{a priori bound}. To this end, let $v_{\calT}\in V^s_{\calT}$ be arbitrary and observe that 
\begin{align}\label{after mon}
\begin{split}
\|v_{\calT}-u_{\calT}\|_{\calT,\lambda}^2 &\leq C_M\left(a_{\calT}(v_{\calT},v_{\calT}-u_{\calT})-a_{\calT}(u_{\calT},v_{\calT}-u_{\calT})  \right)\\ &=C_M a_{\calT}(v_{\calT},v_{\calT}-u_{\calT}) 
\end{split}
\end{align}
by strong monotonicity (A2) and the solution property \eqref{Abstract numerical scheme} of $u_{\calT}$. In order to further bound the right-hand side, we successively use the discrete consistency (A4), the solution property and regularity of $u$, and the Lipschitz property \eqref{Fgamma prop} of $F_{\gamma}$ to obtain
\begin{align*}
a_{\calT}(v_{\calT},v_{\calT}-u_{\calT})&\leq \left\lvert \int_Y F_{\gamma}[v_{\calT}]L_{\calT}(v_{\calT}-u_{\calT})\right\rvert+C_2\lvert v_{\calT}\rvert_{J,\calT}\|v_{\calT}-u_{\calT}\|_{\calT,\lambda}\\&\leq \left(C_1\|F_{\gamma}[v_{\calT}]-F_{\gamma}[u]\|_{L^2(Y)}+C_2\lvert v_{\calT}-u\rvert_{J,\calT}\right)\|v_{\calT}-u_{\calT}\|_{\calT,\lambda}\\ &\leq \left(C_1\left(\sqrt{1-\delta}+\sqrt{n+1}\right)+C_2\right)\| v_{\calT}-u\|_{\calT,\lambda}\|v_{\calT}-u_{\calT}\|_{\calT,\lambda}.
\end{align*}
Combination with the previous estimate \eqref{after mon} yields
\begin{align*}
\|v_{\calT}-u_{\calT}\|_{\calT,\lambda}\leq C_M\left(C_1\left(\sqrt{1-\delta}+\sqrt{n+1}\right)+C_2\right)\| u-v_{\calT}\|_{\calT,\lambda},
\end{align*}
which in turn implies
\begin{align*}
\|u-u_{\calT}\|_{\calT,\lambda}\leq \|u-v_{\calT}\|_{\calT,\lambda}+\|v_{\calT}-u_{\calT}\|_{\calT,\lambda}\leq C_e\|u-v_{\calT}\|_{\calT,\lambda}
\end{align*}
with $C_e>0$ given by \eqref{constant Ce}. We conclude the proof by taking the infimum over $v_{\calT}\in V^s_{\calT}$.
\end{proof}

We conclude this section by noting that Theorem \ref{Tm: a priori bound} implies convergence of the numerical approximation under mesh-refinement. While convergence together with optimal rates follow immediately from standard approximation arguments in the case that the exact solution satisfies additional regularity assumptions, it is not that clear when we only have a minimal regularity solution $u\in H^2_{\mathrm{per}}(Y)$. For the latter case, we can argue as in \cite[Corollary 4.7]{KS21} and obtain the following result.

\begin{rmrk}[Convergence of the numerical approximation]
For a sequence of conforming simplicial meshes $\{\calT_k\}_k$ with $\max_{K\in\calT_k} h_K\rightarrow 0$ as $k\rightarrow\infty$, we have that
\begin{align*}
\inf_{v_{\calT_k}\in V^s_{\calT_k}}\|u-v_{\calT_k}\|_{\calT_k,\lambda}\underset{k\rightarrow\infty}{\longrightarrow} 0.
\end{align*}
In particular, in view of \eqref{a priori bound}, given $a_{\calT_k}:V^s_{\calT_k}\times V^s_{\calT_k}\rightarrow \R$ satisfying (A1)--(A4) with constants uniformly bounded in $k$, we have that
\begin{align*}
\|u-u_{\calT_k}\|_{\calT_k,\lambda}\underset{k\rightarrow\infty}{\longrightarrow} 0
\end{align*}
for the sequence of numerical approximations $\{u_{\calT_k}\}_k\subset V^s_{\calT_k}$.
\end{rmrk}

\subsubsection{The family of numerical schemes}\label{Sec family of num schemes}

For chosen $s\in\{0,1\}$ and a parameter $\theta\in[0,1]$, we now consider the numerical scheme of finding $u_{\calT}\in V^s_{\calT}$ satisfying \eqref{Abstract numerical scheme} with 
\begin{align*}
a_{\calT}:V^s_{\calT}\times V^s_{\calT}\rightarrow\R,\qquad a_{\calT}(w_{\calT},v_{\calT}):=\int_Y F_{\gamma}[w_{\calT}]L_{\lambda,\calT}v_{\calT} + \theta S_{\calT}(w_{\calT},v_{\calT})+J_{\calT}(w_{\calT},v_{\calT}),
\end{align*}
where we define the linear operator $L_{\lambda,\calT}v_{\calT}:=\lambda v_{\calT}-\Delta v_{\calT}$ for $v_{\calT}\in V_{\calT}^s$, the stabilization bilinear form $S_{\calT}:V^s_{\calT}\times V^s_{\calT}\rightarrow\R$ via
\begin{align*}
S_{\calT}(w_{\calT},v_{\calT}):=&\int_Y \left(\nabla^2 w_{\calT}:\nabla^2 v_{\calT}-\Delta w_{\calT}\Delta v_{\calT}\right)+ \int_{\calF} \left(\{\Delta_T w_{\calT}\}\llbracket \nabla v_{\calT}\cdot n\rrbracket + \{\Delta_T v_{\calT}\}\llbracket \nabla w_{\calT}\cdot n\rrbracket\right)\\ &-\int_{\calF} \left( \nabla_T\{\nabla w_{\calT}\cdot n\}\cdot \llbracket \nabla_T v_{\calT}\rrbracket + \nabla_T\{\nabla v_{\calT}\cdot n\}\cdot \llbracket \nabla_T w_{\calT}\rrbracket\right),
\end{align*}
and, for chosen parameters $\eta_1,\eta_2>0$, the jump penalization form $J_{\calT}:V^s_{\calT}\times V^s_{\calT}\rightarrow\R$ via
\begin{align*}
J_{\calT}(w_{\calT},v_{\calT}):= \eta_1\int_{\calF}  h_{\calT}^{-1}\llbracket \nabla w_{\calT}\rrbracket\cdot\llbracket\nabla v_{\calT}\rrbracket +\eta_2 \int_{\calF}  h_{\calT}^{-3}\llbracket w_{\calT}\rrbracket\llbracket v_{\calT}\rrbracket.
\end{align*}
Here, the tangential gradient and Laplacian on mesh faces are denoted by $\nabla_T$ and $\Delta_T$.

This scheme is an adaptation of the method presented in \cite{KS21} for the homogeneous Dirichlet problem. The analysis of this method, i.e., the verification of the assumptions (A1)--(A4), is analogous to \cite{KS21} and hence omitted. The main result is the following:
\begin{thrm}
There exist constants $\bar{\eta}_1,\bar{\eta}_2>0$, depending only on $n,\theta_{\calT},\theta,\bar{p}$ and the Cordes parameters $\delta,\lambda$, such that, for any $\theta\in[0,1]$, if $\eta_1\geq \bar{\eta}_1$ and $\eta_2\geq \bar{\eta}_2$, the properties (A1)--(A4) are satisfied and Theorem \ref{Tm: a priori bound} applies.
\end{thrm}  
\begin{rmrk}
The constants $\bar{\eta}_1,\bar{\eta}_2$ and the constant $C_e$ in the near-best approximation bound \eqref{a priori bound} remain bounded as $\lambda\searrow 0$.
\end{rmrk}

\section{Approximation of Effective Hamiltonians to HJBI Operators}\label{Chapter 3}

\subsection{The effective Hamiltonian}\label{Sec3.1}

We start by recalling the definition of the effective Hamiltonian based on the \textit{cell $\sigma$-problem}; see \cite{AB01,AB10,ABM07}.

Let us consider an HJBI operator $F:\R^n\times \R^n\times \R^n\times \R^{n\times n}_{\mathrm{sym}}\rightarrow \R$ given by
\begin{align}\label{HJBI operator}
F(x,y,p,R):=\inf_{\alpha\in\calA}\sup_{\beta\in\calB}\left\{-A^{\alpha\beta}(y):R-b^{\alpha\beta}(x,y)\cdot p-f^{\alpha\beta}(x,y)  \right\}
\end{align}
with $\calA$ and $\calB$ denoting compact metric spaces, and functions 
\begin{align*}
A=(a_{ij})_{1\leq i,j\leq n}&:\R^n\times\calA\times\calB\rightarrow \R^{n\times n}_{\mathrm{sym}}, &(y,\alpha,\beta)&\mapsto A(y,\alpha,\beta)=:A^{\alpha\beta}(y),\\ b=(b_i)_{1\leq i\leq n}&:\R^n\times\R^n\times\calA\times\calB\rightarrow \R^{n}, &(x,y,\alpha,\beta)&\mapsto b(x,y,\alpha,\beta)=:b^{\alpha\beta}(x,y),\\ f&:\R^n\times\R^n\times\calA\times\calB\rightarrow \R, &(x,y,\alpha,\beta)&\mapsto f(x,y,\alpha,\beta)=:f^{\alpha\beta}(x,y)
\end{align*}
satisfying the assumptions stated below in paragraph \ref{Ass on coef}.  

To the HJBI operator \eqref{HJBI operator}, we associate the corresponding \textit{cell $\sigma$-problem}: for fixed $(x,p,R)\in \R^n\times \R^n\times \R^{n\times n}_{\mathrm{sym}}$ and a positive parameter $\sigma>0$, there exists a unique viscosity solution $v^{\sigma}=v^{\sigma}(\cdot\,;x,p,R)\in C(\R^n)$ to the problem
\begin{align}\label{cell sigma-problem}
\left\{\begin{aligned}
\sigma v^{\sigma}+F(x,y,p,R+\nabla_y^2 v^{\sigma})=0\quad\text{for }y\in Y,\\ y\mapsto v^{\sigma}(y;x,p,R)\text{ is $Y$-periodic}.
\end{aligned}\right.
\end{align}
The function $v^{\sigma}(\cdot\,;x,p,R)$ is called an \textit{approximate corrector}.

\begin{dfntn}[Ergodicity and effective Hamiltonian]
Let $F:\R^n\times \R^n\times \R^n\times \R^{n\times n}_{\mathrm{sym}}\rightarrow \R$ be an HJBI operator of the form \eqref{HJBI operator}. 
\begin{itemize}
\item[(i)] We say $F$ is ergodic (in the $y$-variable) at a point $(x,p,R)\in \R^n\times \R^n\times \R^{n\times n}_{\mathrm{sym}}$ if there exists a constant $H(x,p,R)\in\R$ such that 
\begin{align}\label{H deftn}
-\sigma v^{\sigma}(\cdot\,;x,p,R) \underset{\sigma\searrow 0}{\longrightarrow} H(x,p,R)\quad\text{uniformly}.
\end{align}
Further, we call $F$ ergodic if it is ergodic at every $(x,p,R)\in \R^n\times \R^n\times \R^{n\times n}_{\mathrm{sym}}$.
\item[(ii)] If $F$ is ergodic, we call the function
\begin{align*}
H:\R^n\times \R^n\times \R^{n\times n}_{\mathrm{sym}}\rightarrow\R,\qquad (x,p,R)\mapsto H(x,p,R)
\end{align*}
defined via \eqref{H deftn} the effective Hamiltonian corresponding to $F$.
\end{itemize}
\end{dfntn}

The assumptions on the coefficients made in paragraph \ref{Ass on coef} are such that the HJBI operator \eqref{HJBI operator} fits into the framework considered in \cite{ABM07}, which guarantees ergodicity. The corresponding effective Hamiltonian $H:\R^n\times \R^n\times \R^{n\times n}_{\mathrm{sym}}\rightarrow\R$ is automatically continuous and degenerate elliptic, that is, 
\begin{align*}
R_1-R_2\geq 0\quad \Longrightarrow   \quad H(x,p,R_1)\leq H(x,p,R_2)
\end{align*}
for any $x,p\in \R^n$, $R_1,R_2\in \R^{n\times n}_{\mathrm{sym}}$.
\begin{rmrk}
In the periodic homogenization of elliptic and parabolic HJBI equations 
\begin{align*}
u_{\eps}^{\mathrm{e}}+F\left(x,\frac{x}{\eps},\nabla u_{\eps}^{\mathrm{e}},\nabla^2 u_{\eps}^{\mathrm{e}}\right)=0,\qquad \partial_t u_{\eps}^{\mathrm{p}}+F\left(x,\frac{x}{\eps},\nabla_x u_{\eps}^{\mathrm{p}},\nabla^2_x u_{\eps}^{\mathrm{p}}\right)=0,
\end{align*}
posed in a suitable Dirichlet/Cauchy setting, the effective Hamiltonian determines the homogenized equation
\begin{align*}
u_0^{\mathrm{e}}+H\left(x,\nabla u_{0}^{\mathrm{e}},\nabla^2 u_{0}^{\mathrm{e}}\right)=0,\qquad \partial_t u_0^{\mathrm{p}}+H\left(x,\nabla_x u_{0}^{\mathrm{p}},\nabla^2_x u_{0}^{\mathrm{p}}\right)=0;
\end{align*}
see \cite{ABM07,Eva89,Eva92}.
\end{rmrk}

In this setting, having $A=A(y,\alpha,\beta)$ being independent of the state variable $x$, it can be shown that
\begin{align*}
\left\lvert H(x_1,p,R)-H(x_2,p,R)\right\rvert \leq C\lvert x_1-x_2\rvert (1+\lvert p\rvert)+\omega(\lvert x_1-x_2\rvert)\quad\forall x_1,x_2,p\in\R^n, \,R\in \R^{n\times n}_{\mathrm{sym}},
\end{align*}
for some constant $C>0$ and modulus of continuity $\omega$, which guarantees a comparison principle for the effective problem and implies homogenization; see \cite{ABM07}.

\subsubsection{Assumptions on the coefficients}\label{Ass on coef}

We assume that $A=\frac{1}{2}GG^{\mathrm{T}}\in C(\R^n\times\calA\times\calB; \R^{n\times n})$, $b\in  C(\R^n\times \R^n\times\calA\times\calB; \R^n)$ and $f\in  C(\R^n\times \R^n\times\calA\times\calB; \R)$ satisfy the assumptions listed below. 

\begin{itemize}
\item $G,b,f$ are bounded continuous functions on their respective domains.
\item $G=G(y,\alpha,\beta),b=b(x,y,\alpha,\beta)$ are Lipschitz continuous in $(x,y)$, uniformly in $(\alpha,\beta)$.
\item $f=f(x,y,\alpha,\beta)$ is uniformly continuous in $(x,y)$, uniformly in $(\alpha,\beta)$.
\item $G,b,f$ are $Y$-periodic in the fast variable $y$.
\item Uniform ellipticity: $\exists\,\zeta_1,\zeta_2>0:\; \zeta_1 \lvert \xi\rvert^2\leq  A^{\alpha\beta}(y)\xi \cdot \xi\leq \zeta_2 \lvert \xi\rvert^2\quad \forall y,\xi\in \R^n,\, (\alpha,\beta)\in \calA\times\calB$.
\item Cordes condition: There exist constants $\lambda>0$ and $\delta\in (0,1)$ such that \begin{align}\label{Cordes for effH}
\lvert A^{\alpha\beta}(y)\rvert^2+\frac{\lvert b^{\alpha\beta}(x,y)\rvert^2}{2\lambda} + \frac{1}{\lambda^2}\leq \frac{1}{n+\delta}\left(\mathrm{tr}( A^{\alpha\beta}(y)) + \frac{1}{\lambda}  \right)^{2} 
\end{align}
for all $(x,y,\alpha,\beta)\in \R^n\times \R^n\times \calA\times\calB$.
\end{itemize}

\subsection{Approximation of the cell $\sigma$-problem}\label{Sec:Appr of cell sigma}

For fixed $(x,p,R)\in \R^n\times \R^n\times \R^{n\times n}_{\mathrm{sym}}$ and a positive parameter $\sigma\in (0,\bar{\sigma})$ with fixed $\bar{\sigma}>0$, let us consider the cell $\sigma$-problem \eqref{cell sigma-problem} in the rewritten form
\begin{align}\label{vsigmaprob}
\left\{\begin{aligned}
\inf_{\alpha\in\calA}\sup_{\beta\in\calB}\left\{-A^{\alpha\beta}:\nabla^2 v^{\sigma}+\sigma v^{\sigma}-g^{\alpha\beta}_{x,p,R}  \right\}=0\quad\text{for }y\in Y,\\ y\mapsto v^{\sigma}(y;x,p,R)\text{ is $Y$-periodic},
\end{aligned}\right.
\end{align}
where $g_{x,p,R}^{\alpha\beta}:\R^n\rightarrow \R$ is the $Y$-periodic function given by
\begin{align*}
g_{x,p,R}^{\alpha\beta}(y):=g_{x,p,R}(y,\alpha,\beta):=A^{\alpha\beta}(y):R+b^{\alpha\beta}(x,y)\cdot p + f^{\alpha\beta}(x,y)
\end{align*}
for $y\in\R^n$ and $(\alpha,\beta)\in\calA\times\calB$. The following lemma shows that, for any $\sigma>0$, the problem \eqref{vsigmaprob} admits a unique strong solution $v^{\sigma}\in H^2_{\mathrm{per}}(Y)$ and that we have a uniform bound on $\lvert v^{\sigma}\rvert_{H^2(Y)}$.

\begin{lmm}
Assume that the assumptions of Section \ref{Ass on coef} hold and let $(x,p,R)\in \R^n\times \R^n\times \R^{n\times n}_{\mathrm{sym}}$ be fixed. Then, for any $\sigma>0$, there exists a unique periodic strong solution $v^{\sigma}\in H^2_{\mathrm{per}}(Y)$ to the cell $\sigma$-problem \eqref{vsigmaprob}. Furthermore, we have the bound
\begin{align}\label{unifm bd on H2 seminorm}
\lvert v^{\sigma}\rvert_{H^2(Y)} \leq C
\end{align}
with $C>0$ independent of $\sigma$.
\end{lmm}

\begin{proof}
It is straightforward to check that all assumptions of Theorem \ref{thm well-pos} are satisfied. In particular, the problem \eqref{vsigmaprob} satisfies the Cordes condition
\begin{align*}
\lvert A\rvert^2+ \frac{\sigma^2}{\lambda_{\sigma}^2}\leq \frac{1}{n+\delta} \left(\mathrm{tr}( A) + \frac{\sigma}{\lambda_{\sigma}}  \right)^{2}  \quad\text{in } \R^n\times \calA\times\calB,
\end{align*}
where $\lambda_{\sigma}>0$ is defined by $\lambda_{\sigma}:=\sigma \lambda$. Therefore, we find that there exists a unique periodic strong solution $v^{\sigma}\in H^2_{\mathrm{per}}(Y)$ to \eqref{vsigmaprob}. Note that the corresponding renormalization function $\gamma^{\sigma}\in C(\R^n\times\calA\times\calB)$ (see \eqref{gamma first def}) is given by
\begin{align*}
\gamma^{\sigma}:=\frac{\mathrm{tr}( A) + \frac{\sigma}{\lambda_{\sigma}}}{\lvert A\rvert^2+ \frac{\sigma^2}{\lambda_{\sigma}^2}} = \frac{\mathrm{tr}( A) + \frac{1}{\lambda}}{\lvert A\rvert^2+ \frac{1}{\lambda^2}}
\end{align*}
and hence, $\gamma:=\gamma^{\sigma}$ is independent of $\sigma$. The uniform bound \eqref{unifm bd on H2 seminorm} now follows from Remark \ref{Rk: bound on soln}.
\end{proof}

The discontinuous Galerkin ($s=0$) or the $C^0$-IP ($s=1$) finite element method from Section \ref{Section FE schemes for periodic HJBI} yields an approximation $v_{\calT}^{\sigma}\in V^s_{\calT}$ to the problem \eqref{vsigmaprob} satisfying
\begin{align}\label{first ine}
\|v^{\sigma}-v^{\sigma}_{\calT}\|_{\calT,\lambda_{\sigma}}\leq C \inf_{z_{\calT}\in V^s_{\calT}}\|v^{\sigma}-z_{\calT}\|_{\calT,\lambda_{\sigma}}\leq C \inf_{z_{\calT}\in V^s_{\calT}}\|v^{\sigma}-z_{\calT}\|_{\calT,\bar{\sigma}\lambda},
\end{align}
where the constant $C>0$ can be chosen to be independent of $\sigma$; see Section \ref{Sec a priori analysis}.

\begin{lmm}[Approximation of the cell $\sigma$-problem]\label{Thm: Appr of appr corr}
Assume that the assumptions of Section \ref{Ass on coef} hold, and that the periodic strong solution $v^{\sigma}=v^{\sigma}(\cdot\,;x,p,R)\in H^2_{\mathrm{per}}(Y)$ to \eqref{vsigmaprob} satisfies $v^{\sigma}\in H^{2+r_K}(K)$ with $r_K\geq 0$ for all $K\in \calT$. Then, we have the error bound 
\begin{align*}
\|v^{\sigma}-v^{\sigma}_{\calT}\|_{\calT,\lambda_{\sigma}}\lesssim \inf_{z_{\calT}\in V^s_{\calT}}\|v^{\sigma}-z_{\calT}\|_{\calT,\bar{\sigma}\lambda} \lesssim \sqrt{ \sum_{K\in \calT} h_K^{2\min\{r_K,\bar{p}-1\}}\|\nabla v^{\sigma}\|_{H^{1+r_K}(K)}^2}
\end{align*}
with constants independent of $\sigma$ and the choice of $(x,p,R)$.
\end{lmm}

The proof is omitted as the first inequality was already obtained in \eqref{first ine}, while the second estimate is a consequence of standard approximation arguments. 

Let us observe that without any additional regularity assumptions on $v^{\sigma}$, we have that $\|\nabla v^{\sigma}\|_{H^1(Y)}\leq C$ is uniformly bounded in $\sigma$. Indeed, this follows from \eqref{unifm bd on H2 seminorm} and Poincar\'{e}'s inequality.

\subsection{Approximation of the effective Hamiltonian}\label{Sec3.3}

Let us note that the effective Hamiltonian $H:\R^n\times \R^n\times \R^{n\times n}_{\mathrm{sym}}\rightarrow\R$ given by \eqref{H deftn} is defined via the viscosity solutions $v^{\sigma}\in C(\R^n)$ to the cell $\sigma$-problem \eqref{cell sigma-problem}. Therefore, we make the technical assumption that 
\begin{align}\label{tech ass}
v^{\sigma}\in W^{2,n}_{\mathrm{loc}}(\R^n),
\end{align}
so that the strong solution coincides with the unique viscosity solution to \eqref{vsigmaprob}; see \cite{CCK96,Lii83,Lio83}. This is no further restriction when $n=2$ or when we have an HJB problem; see \cite{GSS21}.

Let us define the approximate effective Hamiltonian $H_{\calT}^{\sigma}$ for $\sigma>0$ via
\begin{align}\label{H_T^s defn}
H_{\calT}^{\sigma}:\R^n\times \R^n\times \R^{n\times n}_{\mathrm{sym}}\rightarrow\R,\qquad H_{\calT}^{\sigma}(x,p,R):=-\sigma \int_Y v_{\calT}^{\sigma}(\cdot\,;x,p,R).
\end{align}
We note that this definition is quite natural as we have from \eqref{H deftn} that
\begin{align*}
Q_{x,p,R}^{\sigma}:=\left\|-\sigma v^{\sigma}(\cdot\,;x,p,R)-H(x,p,R)\right\|_{L^{\infty}(Y)}\underset{\sigma\searrow 0}{\longrightarrow} 0
\end{align*}
for any $(x,p,R)\in \R^n\times \R^n\times \R^{n\times n}_{\mathrm{sym}}$.

\begin{thrm}[Approximation of the effective Hamiltonian]\label{Thm: Appr of H}
Assume that the assumptions of Section \ref{Ass on coef} and \eqref{tech ass} hold. Let $H:\R^n\times \R^n\times \R^{n\times n}_{\mathrm{sym}}\rightarrow\R$ denote the effective Hamiltonian given by \eqref{H deftn} and $H_{\calT}^{\sigma}:\R^n\times \R^n\times \R^{n\times n}_{\mathrm{sym}}\rightarrow\R$ its numerical approximation \eqref{H_T^s defn}. Then, for $\sigma\in (0,\bar{\sigma})$ and $(x,p,R)\in \R^n\times \R^n\times \R^{n\times n}_{\mathrm{sym}}$, we have the error bound
\begin{align}\label{Appr of H}
\lvert H_{\calT}^{\sigma}(x,p,R)-H(x,p,R)\rvert \lesssim Q_{x,p,R}^{\sigma}+\inf_{z_{\calT}\in V^s_{\calT}}\|v^{\sigma}(\cdot\,;x,p,R)-z_{\calT}\|_{\calT,\bar{\sigma}\lambda}.
\end{align}
In particular, we have the following assertions.
\begin{itemize}
\item[(i)] If there exist $\{r_K\}_{K\in\calT}\subset [0,\infty)$ such that $\sup_{K\in\calT} \|\nabla v^{\sigma}(\cdot\,;x,p,R)\|_{H^{1+r_K}(K)}\leq C_{x,p,R} \lvert K\rvert^{\frac{1}{2}}$ holds uniformly in $\sigma$, then we have that
\begin{align}\label{Appr of H2}
\lvert H_{\calT}^{\sigma}(x,p,R)-H(x,p,R)\rvert \lesssim Q_{x,p,R}^{\sigma}+C_{x,p,R}\sqrt{\sum_{K\in \calT} h_K^{2\min\{r_K,\bar{p}-1\}} \lvert K\rvert }.
\end{align}
\item[(ii)] If there exists $r\geq 0$ such that $\|\nabla v^{\sigma}(\cdot\,;x,p,R)\|_{H^{1+r}(Y)}\leq C_{x,p,R}$ holds uniformly in $\sigma$, then we have that
\begin{align}\label{Appr of H3}
\lvert H_{\calT}^{\sigma}(x,p,R)-H(x,p,R)\rvert \lesssim Q_{x,p,R}^{\sigma}+C_{x,p,R}\,h^{\min\{r,\bar{p}-1\}},
\end{align}
where we write $h:=\max_{K\in\calT}h_K$.
\end{itemize}
The constants absorbed in $\lesssim$ in the above estimates \eqref{Appr of H}, \eqref{Appr of H2} and \eqref{Appr of H3} are independent of $\sigma$ and $(x,p,R)$.
\end{thrm}

\begin{proof}
Let $\sigma\in (0,\bar{\sigma})$ and $(x,p,R)\in \R^n\times \R^n\times \R^{n\times n}_{\mathrm{sym}}$. We observe that by Lemma \ref{Thm: Appr of appr corr}, and recalling $\lambda_{\sigma}=\sigma\lambda$, we have
\begin{align}\label{for tri 1}
\begin{split}
\left\|\sigma v^{\sigma}(\cdot\,;x,p,R)-\sigma v^{\sigma}_{\calT}(\cdot\,;x,p,R)\right\|_{L^2(Y)} &\lesssim \|v^{\sigma}(\cdot\,;x,p,R)-v^{\sigma}_{\calT}(\cdot\,;x,p,R)\|_{\calT,\lambda_{\sigma}}\\ &\lesssim \inf_{z_{\calT}\in V^s_{\calT}}\|v^{\sigma}(\cdot\,;x,p,R)-z_{\calT}\|_{\calT,\bar{\sigma}\lambda}
\end{split}
\end{align}
with constants independent of $\sigma$ and $(x,p,R)$. Further, we note that
\begin{align}\label{for tri 2}
\|-\sigma v^{\sigma}(\cdot\,;x,p,R) - H(x,p,R)\|_{L^2(Y)}\leq Q_{x,p,R}.
\end{align}
We can now conclude, using H\"{o}lder and triangle inequalities together with \eqref{for tri 1} and \eqref{for tri 2}, that we have
\begin{align*}
\lvert H_{\calT}^{\sigma}(x,p,R)-H(x,p,R)\rvert &=\left\lvert  -\sigma\int_Y   v_{\calT}^{\sigma}(\cdot\,;x,p,R)- H(x,p,R)  \right\rvert \\&= \left\lvert  \int_Y \left(-\sigma v_{\calT}^{\sigma}(\cdot\,;x,p,R) - H(x,p,R)\right)  \right\rvert \\ &\leq \|-\sigma v_{\calT}^{\sigma}(\cdot\,;x,p,R) - H(x,p,R)\|_{L^2(Y)}\\ &\lesssim Q_{x,p,R}+\inf_{z_{\calT}\in V^s_{\calT}}\|v^{\sigma}(\cdot\,;x,p,R)-z_{\calT}\|_{\calT,\bar{\sigma}\lambda},
\end{align*}
where the constant absorbed in $\lesssim$ is independent of $\sigma$ and $(x,p,R)$. This completes the proof of \eqref{Appr of H}. The assertions (i) and (ii) are immediate consequences of \eqref{Appr of H} in view of Lemma \ref{Thm: Appr of appr corr}.
\end{proof}

\begin{rmrk}[Improvement for HJB operators]\label{Rk:Improvement for HJB}
Let us assume that the coefficients $A,b,f$ from the HJBI operator \eqref{HJBI operator} are such that the operator simplifies to an HJB operator
\begin{align*}
F(x,y,p,R):=\sup_{\beta\in\calB}\left\{-A(y,\beta):R-b(x,y,\beta)\cdot p-f(x,y,\beta)  \right\}
\end{align*}
with $f$ satisfying the same assumptions as the components of $b$. We then have for $\sigma\in (0,\bar{\sigma})$ with $\bar{\sigma}$ sufficiently small and $(x,p,R)\in \R^n\times \R^n\times \R^{n\times n}_{\mathrm{sym}}$ that $Q_{x,p,R}^{\sigma}\leq C \sigma\left(1+\lvert p\rvert +\lvert R\rvert \right)$ and $\|\nabla v^{\sigma}\|_{H^{1+r}(Y)}\leq C(1+\lvert p\rvert +\lvert R\rvert)$, uniformly in $\sigma$, for some $r>0$; see \cite{CM09,GSS21}. Therefore, by Theorem \ref{Thm: Appr of H} (ii), we have the error bound
\begin{align*}
\lvert H_{\calT}^{\sigma}(x,p,R)-H(x,p,R)\rvert \lesssim \left(\sigma+h^{\min\{r,\bar{p}-1\}}\right)\left(1+\lvert p\rvert +\lvert R\rvert \right),
\end{align*}
where the constant absorbed in $\lesssim$ is independent of $\sigma$ and $(x,p,R)$.
\end{rmrk}

\section{Numerical Experiments}\label{Chap4}

\subsection{Numerical solution of a periodic HJBI problem}\label{Sec4.1}

In this numerical experiment, we consider the periodic HJBI problem
\begin{align}\label{Experiment 1 problem}
\left\{\begin{aligned}
\inf_{\alpha\in [0,\frac{1}{2}]}\sup_{\beta\in [0,2\pi]}\left\{ -A^{\alpha\beta}:\nabla^2 u  + c^{\alpha\beta} u -f^{\alpha\beta}\right\} = 0 \quad \text{in }Y,\\ u \text{ is $Y$-periodic},
\end{aligned}\right.
\end{align}
where we define the diffusion coefficient by
\begin{align*}
A^{\alpha\beta}:=Q(\beta)\begin{pmatrix}
\frac{\cos(\alpha)+\sin(\alpha)}{\sqrt{2}} & 0\\0 & \frac{\cos(\alpha)-\sin(\alpha)}{\sqrt{2}}
\end{pmatrix}Q(\beta)^{\mathrm{T}},\qquad Q(\beta):=\begin{pmatrix}
\cos(\beta) & -\sin(\beta)\\ \sin(\beta) & \cos(\beta)
\end{pmatrix},
\end{align*}
and set $c^{\alpha\beta}:=\frac{\sec(\alpha)}{\sqrt{2}}$ and $f^{\alpha\beta}:=\frac{\sec(\alpha)}{\sqrt{2}}\tilde{f}$  for $(\alpha,\beta)\in [0,\frac{1}{2}]\times [0,2\pi]$. Here, we choose $\tilde{f}\in C_{\mathrm{per}}(Y)$ such that the solution to \eqref{Experiment 1 problem} is given by
\begin{align*}
u:[0,1]^2\rightarrow\R,\qquad u(y_1,y_2)=\cos(2\pi y_1)\cos(2\pi y_2).
\end{align*}
We leave it to the reader to check that this problem fits into the setting of Section \ref{Section 2 setting}. In particular, we have that the Cordes condition \eqref{Cordes periodic} holds with $\lambda=1$.

\begin{rmrk}
The renormalized HJBI problem \eqref{HJBI renormalized} corresponding to \eqref{Experiment 1 problem} is given by
\begin{align*}
\left\{\begin{aligned}
\inf_{\alpha\in [0,\frac{1}{2}]}\sup_{\beta\in [0,2\pi]}\left\{ -\gamma^{\alpha\beta} A^{\alpha\beta}:\nabla^2 u  +  u \right\} = \tilde{f} \quad \text{in }Y,\\ u \text{ is $Y$-periodic},
\end{aligned}\right.
\end{align*}
where $\gamma^{\alpha\beta}:=\sqrt{2}\cos(\alpha)$ for $(\alpha,\beta)\in [0,\frac{1}{2}]\times [0,2\pi]$. 
\end{rmrk}

We apply the $C^0$-IP and discontinuous Galerkin finite element schemes from Section \ref{Sec family of num schemes} to the HJBI problem \eqref{Experiment 1 problem}. Under uniform mesh-refinement, we illustrate the behavior of the error 
\begin{align}\label{T error}
\|u-u_{\calT}\|_{\calT}:=\sqrt{
\int_Y \left(\lvert \nabla^2 (u-u_{\calT})\rvert^2 +2\lvert \nabla (u-u_{\calT})\rvert^2+  (u-u_{\calT})^2\right) + \lvert u-u_{\calT}\rvert_{J,\calT}^2}
\end{align} 
and of the \textit{a posteriori} error estimator (see Theorem \ref{thm: a posteriori}), i.e., 
\begin{align}\label{eta error}
\eta_{\calT}(u_{\calT}):= \sqrt{\int_Y \lvert F_{\gamma}[u_{\calT}]\rvert^2 + \lvert u_{\calT}\rvert_{J,\calT}^2}
\end{align}
for the numerical approximation $u_{\calT}\in V^s_{\calT}$. For the implementation, we have used the software package NGSolve \cite{Sch14} and the discrete nonlinear problems are solved using a Howard-type algorithm as in \cite{KS21}. Figure \ref{fig:CG and DG} presents the performance of the $C^0$ interior penalty and discontinuous Galerkin finite element methods using polynomial degrees $\bar{p}\in\{2,3\}$ and parameters $\theta\in\{0,\frac{1}{2}\}$. We observe optimal rates of convergence for both schemes, that is, order $\calO(N^{-\frac{1}{2}})$ for $\bar{p}=2$ and order $\calO(N^{-1})$ for $\bar{p}=3$, where we denote the number of degrees of freedom by $N$.

\begin{figure}
\centering
\begin{subfigure}{.5\textwidth}
  \centering
  \includegraphics[width=\linewidth]{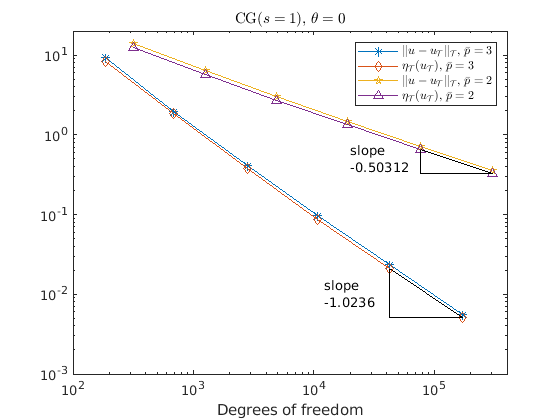}
\end{subfigure}%
\begin{subfigure}{.5\textwidth}
  \centering
  \includegraphics[width=\linewidth]{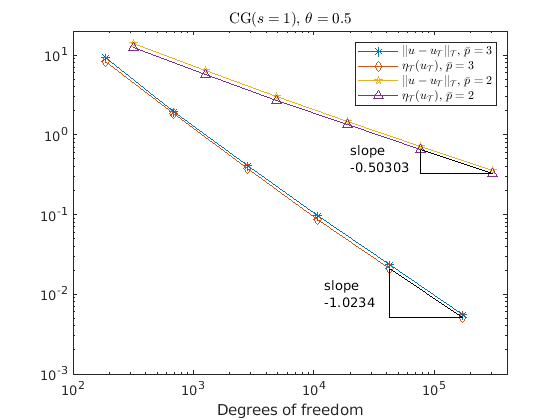}
\end{subfigure}\\
\begin{subfigure}{.5\textwidth}
  \centering
  \includegraphics[width=\linewidth]{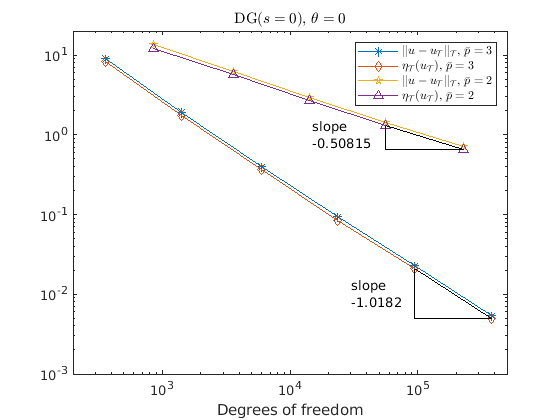}
\end{subfigure}%
\begin{subfigure}{.5\textwidth}
  \centering
  \includegraphics[width=\linewidth]{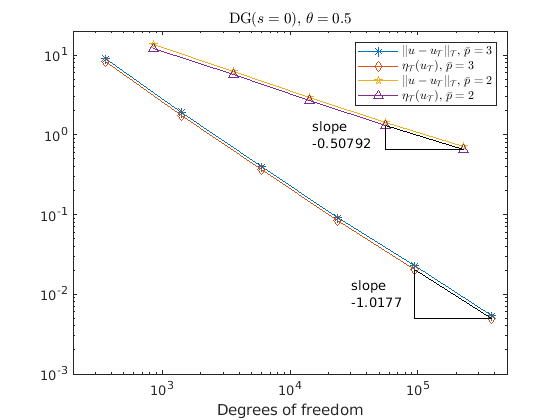}
\end{subfigure}
\caption{Approximation of the solution $u$ to the HJBI problem \eqref{Experiment 1 problem} via the $C^0$-IP (top) and DG (bottom) schemes under mesh-refinement with polynomial degrees $\bar{p}\in\{2,3\}$. We illustrate the error \eqref{T error} and the \textit{a posteriori} estimator \eqref{eta error} for the approximation $u_{\calT}\in V^s_{\calT}$ to the solution $u$, using $\theta=0$ (left) and $\theta=\frac{1}{2}$ (right).}
\label{fig:CG and DG}
\end{figure}

\subsection{Numerical approximation of the effective Hamiltonian}\label{Sec 4.2}

In this numerical experiment, we demonstrate the numerical scheme for the approximation of the effective Hamiltonian corresponding to the HJBI operator
\begin{align}\label{HJBI Experiment 2}
F:\R^2\times \R^{2\times 2}_{\mathrm{sym}}\rightarrow \R,\qquad F(y,R) := \inf_{\alpha\in \calA} \sup_{\beta\in \calB} \left\{ -A^{\alpha\beta}(y):R  - 1 \right\}  
\end{align}
with $\calA:=[1,2]$, $\calB:=[0,1]$, and the coefficient $A=A(y,\alpha,\beta):\R^2\times\calA\times\calB\rightarrow \R^{2\times 2}_{\mathrm{sym}}$ given by
\begin{align*}
A^{\alpha\beta}(y):= \left(a_0(y)+\alpha\beta a_1(y)\right) B,
\end{align*}
where we choose positive scalar functions $a_0,a_1:\R^2\rightarrow (0,\infty)$ and a symmetric positive definite matrix $B\in \R^{2\times 2}_{\mathrm{sym}}$ defined by 
\begin{align*}
B:=\begin{pmatrix}
2 & -1\\-1 & 4
\end{pmatrix},\qquad a_0\equiv 1,\qquad a_1(y):=\sin^2(2\pi y_1)\cos^2(2\pi y_2)+1.
\end{align*}

It is straightforward to check that this problem fits into the framework of Section \ref{Ass on coef} and in particular we have that the Cordes condition \eqref{Cordes for effH} holds with $\lambda = \frac{1}{4}$. This HJBI operator is chosen so that we know the effective Hamiltonian explicitly.

\begin{rmrk}
It can be checked that the HJBI operator \eqref{HJBI Experiment 2} can be rewritten as HJB operator
\begin{align*}
F(y,R)=\sup_{\beta\in \calB} \left\{ -\left(a_0(y)+\beta a_1(y)\right) B:R  - 1 \right\},\qquad (y,R)\in \R^2\times \R^{2\times 2}_{\mathrm{sym}},
\end{align*}
for which the effective Hamiltonian $H:\R^{2\times 2}_{\mathrm{sym}}\rightarrow \R$ is known explicitly and given by 
\begin{align*}
H(R):=\max\left\{ -\left(\int_Y \frac{1}{a_0}\right)^{-1} B:R - 1, -\left(\int_Y \frac{1}{a_0+a_1}\right)^{-1}B:R - 1\right\}
\end{align*}
for $R\in \R^{2\times 2}_{\mathrm{sym}}$; see \cite{FO18}.
\end{rmrk}

We make it our goal to approximate the effective Hamiltonian $H(R)$ at the point
\begin{align*}
R:=\begin{pmatrix}
-2 & 1\\1 & -3
\end{pmatrix},
\end{align*}
noting that the same problem was already used for the numerical experiments in \cite{GSS21}. As we have $B:R = -18 < 0$, the true effective Hamiltonian at this chosen point can be computed as
\begin{align}\label{H(R) true value}
H(R)=-\left(\int_Y \frac{1}{a_0+a_1}\right)^{-1}B:R - 1 = \frac{9\sqrt{6}\pi}{K(\frac{1}{3})}-1\approx 38.9429127,                
\end{align}
where $K$ denotes the complete elliptic integral of the first kind. 

In our numerical experiments, we approximate the true value of the effective Hamiltonian $H(R)$ from \eqref{H(R) true value} by $H_{\calT}^{\sigma}(R)$ as defined in \eqref{H_T^s defn}, where we use the $C^0$-IP finite element method ($s=1$) with $\theta = \frac{1}{2}$ to obtain the approximation $v_{\calT}^{\sigma}(\cdot\,;R)$ to the solution $v^{\sigma}(\cdot\,;R)$ of the cell $\sigma$-problem as described in Section \ref{Sec:Appr of cell sigma}. We denote the relative approximation error by
\begin{align*}
E_{\calT}^{\sigma}:=\frac{\lvert H_{\calT}^{\sigma}(R) - H(R) \rvert}{\lvert H(R)\rvert},\qquad H_{\calT}^{\sigma}(R):=-\sigma \int_Y v_{\calT}^{\sigma}(\cdot\,;R)
\end{align*}
and further write 
\begin{align*}
E^{\sigma}:= \frac{\lvert H^{\sigma}(R) - H(R) \rvert}{\lvert H(R)\rvert},\qquad H^{\sigma}(R):=-\sigma \int_Y v^{\sigma}(\cdot\,;R).
\end{align*}
Let us point out that the approximate corrector $v^{\sigma}(\cdot\,;R)$ and consequently the value of $E^{\sigma}$ is not known exactly, but we expect that $E^{\sigma}=\calO(\sigma)$ from Remark \ref{Rk:Improvement for HJB}.

\begin{figure}
\centering
\begin{subfigure}{.5\textwidth}
  \centering
  \includegraphics[width=\linewidth]{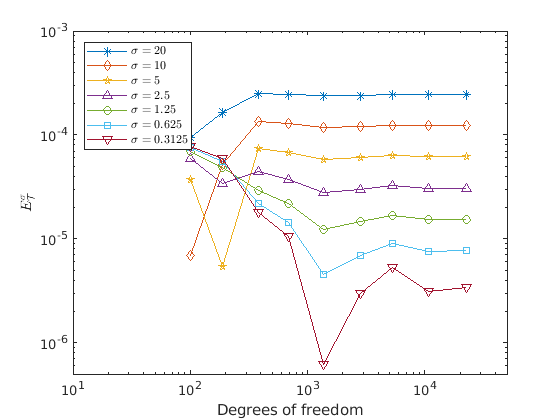}
\end{subfigure}%
\begin{subfigure}{.5\textwidth}
  \centering
  \includegraphics[width=\linewidth]{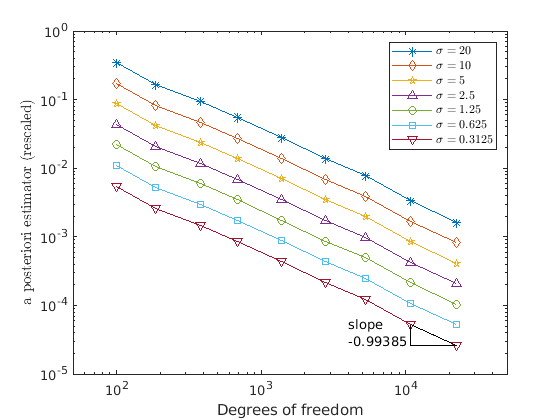}
\end{subfigure}\\
\begin{subfigure}{.5\textwidth}
  \centering
  \includegraphics[width=\linewidth]{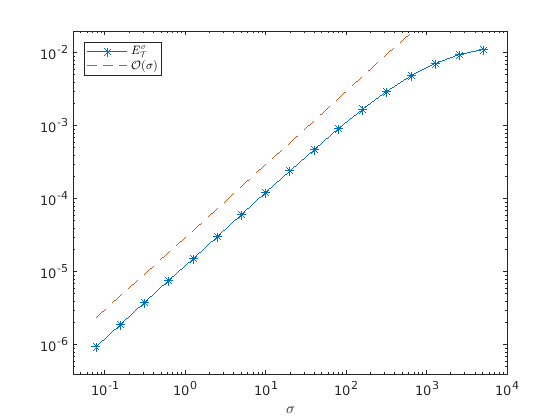}
\end{subfigure}%
\begin{subfigure}{.5\textwidth}
  \centering
  \includegraphics[width=\linewidth]{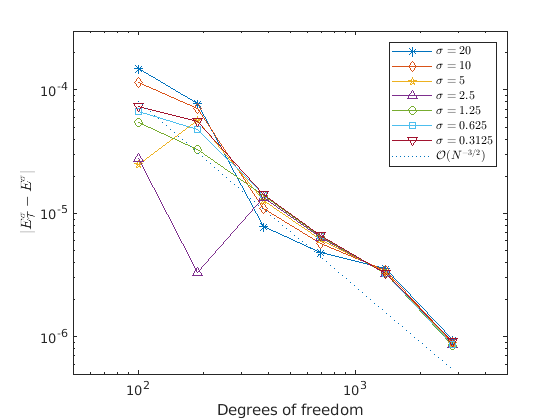}
\end{subfigure}
\caption{Top: Relative error $E_{\calT}^{\sigma}$ (left) and rescaled \textit{a posteriori} error estimator (right) for fixed $\sigma$ under mesh-refinement using $\bar{p}=3$. Bottom:  Relative error $E_{\calT}^{\sigma}$ using a fixed discretization with $\bar{p}=20$ (left), and illustration of the speed of convergence of $E_{\calT}^{\sigma}$ to $E^{\sigma}$ for fixed $\sigma$ under mesh-refinement using $\bar{p}=3$ (right).} 
\label{fig:effHam_vary_h}
\end{figure}

Figure \ref{fig:effHam_vary_h} (top) shows the behavior of the relative approximation error $E_{\calT}^{\sigma}$ under uniform mesh-refinement for fixed values of $\sigma$, and the corresponding \textit{a posteriori} error estimator $\eta_{\calT}(v_{\calT}^{\sigma})$ (re-scaled by a multiplicative constant $C_{\sigma}$ for illustration purposes) using polynomial degree $\bar{p}=3$. We observe that $E_{\calT}^{\sigma}$ converges to a constant, namely $E^{\sigma}$, and that the \textit{a posteriori} estimator is of order $\calO(N^{-1})$ as expected, where $N$ denotes the degrees of freedom. In particular, let us emphasize that this is the expected behavior and that the relative error for large numbers of degrees of freedom is entirely dominated by the $\sigma$-error $E^{\sigma}$. 

Figure \ref{fig:effHam_vary_h} (bottom) illustrates accurate approximations to the unknown values $E^{\sigma}$ for various values of the parameter $\sigma$, and the convergence rate for the convergence of $E_{\calT}^{\sigma}$ to the value $E^{\sigma}$. The accurate approximations to the values $E^{\sigma}$ are obtained using high polynomial degree $\bar{p}=20$ and a fixed triangulation (longest edge $\sqrt{2}\times 2^{-3}$), and we observe convergence of order $\calO(\sigma)$ as $\sigma$ tends to zero, as expected. Let us note that it is difficult to obtain accurate approximations for extremely small values of $\sigma$ as those lead to poorly conditioned discrete problems. We further observe that $\lvert E_{\calT}^{\sigma} - E^{\sigma}\rvert$ is of order $\calO(N^{-\frac{3}{2}})$ for fixed $\sigma$, where we take the unknown value $E^{\sigma}$ to be the previously obtained accurate approximation. This rate is higher than predicted by Remark \ref{Rk:Improvement for HJB}, which is based on an error estimate in the $\|\cdot\|_{\calT,\lambda_{\sigma}}$-norm and is therefore indeed expected to overestimate the error between $H_{\calT}^{\sigma}(R)$ and $H(R)$ related to the weaker integral functional from \eqref{H_T^s defn}.

\section{Conclusion}

In this work we introduced discontinuous Galerkin and $C^0$ interior penalty finite element schemes for the numerical approximation of periodic HJBI problems with an application to the approximation of effective Hamiltonians to ergodic HJBI operators. The first part of this paper was focused on periodic HJBI cell problems and we have performed rigorous \textit{a posteriori} and \textit{a priori} error analyses for a wide class of numerical schemes. In particular, the \textit{a posteriori} analysis was independent of the choice of numerical scheme. The second part of this paper was focused on the approximation of the effective Hamiltonian corresponding to ergodic HJBI operators. An approximation scheme for the effective Hamiltonian via a DG/$C^0$-IP approximation to approximate correctors was presented and rigorously analyzed. Finally, we presented numerical experiments illustrating the theoretical results and the performance of the numerical schemes.

\section*{Acknowledgments}

The work of TS was supported by the UK Engineering and Physical Sciences Research Council [EP/L015811/1]. We gratefully acknowledge helpful conversations with Professor Dietmar Gallistl (Friedrich-Schiller-Universit\"{a}t Jena) and Professor Endre S\"{u}li (University of Oxford) during the preparation of this work.

\bibliographystyle{plain}
\bibliography{refHJBI}

\end{document}